\documentclass[11pt, twoside]{article}
\usepackage{graphicx}
\usepackage[caption=false]{subfig}
\usepackage{array}
\usepackage{threeparttable} 
\usepackage{amsmath}
\usepackage{amssymb,amsfonts}
\usepackage{amsthm}
\usepackage{bm}
\usepackage{mathrsfs}
\usepackage{amssymb}
\usepackage{multirow}
\usepackage[margin=1in]{geometry}
\usepackage{algorithm}
\usepackage{algorithmic}
\numberwithin{equation}{section}
\usepackage{multirow}
\usepackage{makecell}
\usepackage{booktabs}
\theoremstyle{definition}
\newtheorem{theorem}{Theorem}

\newtheorem{lemma}{Lemma}

\newtheorem{remark}{Remark}
\newtheorem{example}{Example}
\newtheorem{assumption}{Assumption}

\usepackage{cite}
\usepackage{hyperref}
\usepackage[nameinlink]{cleveref}
\newcommand{\vertiii}[1]{{\left\vert\kern-0.25ex\left\vert\kern-0.25ex\left\vert #1 
		\right\vert\kern-0.25ex\right\vert\kern-0.25ex\right\vert}}

\usepackage{fancyhdr}
\begin{document}
	
	\title{An Incremental SVD Method for Non-Fickian Flows in Porous Media: Addressing Storage and Computational Challenges}%\tnoteref{mytitlenote}}

\author{Gang Chen%
	\thanks{ College of  Mathematics, Sichuan University, Chengdu 610064, China (\mbox{cglwdm@scu.edu.cn},\mbox{zuodujin@stu.scu.edu.cn}).}
	\and
	Yangwen Zhang%
	\thanks{Department of Mathematics, University of Louisiana at Lafayette, Lafayette, LA 70503, USA  (\mbox{yangwen.zhang@louisiana.edu}).}
	\and
	Dujin Zuo		\footnotemark[1]}

\date{\today}

\maketitle

\begin{abstract}
 	It is well known that the numerical solution of the Non-Fickian flows at the current stage depends on all previous time instances. Consequently, the storage requirement increases linearly, while the computational complexity grows quadratically with the number of time steps. This presents a significant challenge for numerical simulations.   While numerous existing methods address this issue, our proposed approach stems from a data science perspective and maintains uniformity. Our method relies solely on the rank of the solution data, dissociating itself from dependency on any specific partial differential equation (PDE). In this paper, we make the assumption that the solution data exhibits approximate low rank. Here, we present a memory-free algorithm, based on the incremental SVD technique, that exhibits only linear growth in computational complexity as the number of time steps increases. We prove that the error between the solutions generated by the conventional algorithm and our innovative approach lies within the scope of machine error.  Numerical experiments are showcased to affirm the accuracy and efficiency gains in terms of both memory usage and computational expenses.
\end{abstract}

\section{Introduction}
The non-Fickian ﬂow of ﬂuid in porous media \cite{MR1951906} is complicated by the history eﬀect which characterizes various mixing length growth of the ﬂow and can be modeled by an integro-diﬀerential equation: Find $u = u(x,t)$ such that
\begin{subequations}\label{integro-diff-Eq}
	\begin{align}
		u_t+\mathcal Au+\int_0^t K(t-s) \mathcal Bu(s)\ {\rm d}s&=f(x,t),\quad \text{in } \Omega\times \left( 0,T\right], \\
		u&=0,\qquad\quad~ \text{on } \partial \Omega\times \left( 0,T\right],\\
		u(x,0)&=u_0(x),\quad~ \text{in } \Omega,
	\end{align}
\end{subequations}
where $\Omega\in \mathbb{R}^d (d = 2,3)$ is a bounded convex polygonal domain or polyhedral domain  with  boundary $\partial\Omega$, $u_0$ is a given function defined on $\Omega$, $K(t)$ is a nonnegative memory kernel, singular or non-singular, $f(x,t)$ is a known function,  $\mathcal A$ is a symmetric positive definite second-order elliptic operator  and is of the form:
\begin{align*}
	\mathcal A=-\sum_{i,j=1}^d\frac{\partial }{\partial x_i}(a_{ij}(x)\frac{\partial}{\partial x_j})&+a(x)I, \quad a(x)\ge 0,\\
	a_{ij}(x)=a_{ji}(x),  i,j=1,\cdots,d,  a_1\sum_{i=1}^{d}\xi_i^2&\ge \sum_{i,j=1}^{d}a_{ij}\xi_i\xi_j\ge a_0\sum_{i=1}^{d}\xi_i^2,  a_1\geq a_0> 0.
\end{align*}
The operator $\mathcal B$ is any second-order linear operator and takes the following form:
\begin{align*}
	\mathcal B=-\sum_{i,j=1}^d\frac{\partial }{\partial x_i}(b_{ij}(x)\frac{\partial}{\partial x_j})+\sum_{i=1}^{d}b_i(x)\frac{\partial}{\partial x_i}+b(x)I.
\end{align*}

Numerous numerical approaches have been put forth to address the problem  \eqref{integro-diff-Eq}. Among the array of computational techniques available, finite difference for time discretization and Galerkin finite element for spatial discretization have gained significant prominence. Time discretization methods encompass strategies rooted in backward Euler, second-order backward differentiation formula (BDF2) scheme, Crank-Nicolson, Richard extrapolation, Runge-Kutta methods, see \cite{WangWansheng2018Teaf,LiXiaoli2016Atbf,YangXuehua2013Cmfs,LuoMan2013CQBN,ChenYanping2012Read,OstermannAlexander2023EERM,MR3614283,MR2608473,MR2608473,QiuWenlin2023AFEG}   and references therein.  Spatial discretization techniques encompass a range of methodologies, including conventional finite element methods \cite{MR1686149,MR2670111,MR3335208,MR2608473,MR1897408,MR1393903,MR3958333}, mixed finite element methods \cite{MR1951906,MR2447252,MR1713237,MR2551194}, finite volume methods \cite{MR2206438}, and discontinuous Galerkin methods \cite{MR2272610}. Further information can be found in references \cite{MR3778339,MR3557149,MR1034918,MR2580558,MR859017,MR1220827} and the sources cited within them.

One key point of  numerical simulation for \eqref{integro-diff-Eq} is  how to discretize the memory kernel term. Among extensive existing literatures, there are plenty of method devised for  discretizing some special kernel , for example, $ K(t)=\rho e^{-\delta t}  $ in \cite{MR3778339,MR3557149} . Discretizing this kind of kernel is usually convenient and easily to computed , since one can take advantage of the preceding time numerical solutions fully by devising some specific fully discretization scheme and therefore avoiding storing these preceding numerical solutions. However, for some special kernel , such as nonsingular kernel $ K(t)=\log{(1+t)} $ or some weakly singular kernel $ K(t)= e^{-\lambda t}\frac{1}{\Gamma(\alpha)}t^{\alpha-1}$, or  variable-order Caputo fractional kernel, it becomes necessary to store all preceding time numerical solutions when computing the solution at the next step. For general direct method, the storage requirement increases linearly while the computational complexity grows quadratically with the number of time steps. This poses a substantial challenge for numerical simulations. To address this issue, numerous fast algorithms have been proposed to reduce storage costs and enhance computational performance. For instance, Brunner et al. \cite{LiangHui2019Tcoc,LiangHui2020Cmfi,J.-P.Kauthen1997Ccat,LIANGHUI2016IETO,MR3084175} propose a collocation method for solving general systems of linear integral-algebraic equations (IAEs).  Additionally, Lubich et al. propose convolution quadrature as an approach to approximate the integral term of the memory kernel \cite{MR2231714,EduardoCuesta2006Cqtd,LubichC.1988Cqad,MR0923707}. Jiang et al. \cite{MR3639246} propose extensive fast algorithms  for dealing with Caputo derivative and make good progress. 
For further rapid algorithms, refer to \cite{KwonKiwoon2003Apmf,MR1163355,HuangYuxiang2022Eeot,MR3958333,GuoLing2019Emmf} and the references therein.

%One challenge faced by these numerical schemes is the need to store all preceding time numerical solutions when computing the solution at the next step. The storage requirement increases linearly while the computational complexity grows quadratically with the number of time steps, see more details in \Cref{sec2}. This poses a substantial challenge for numerical simulations. To address this issue, numerous fast algorithms have been proposed to reduce storage costs and enhance computational performance. For instance, Brunner et al. \cite{LiangHui2019Tcoc,LiangHui2020Cmfi,J.-P.Kauthen1997Ccat,LIANGHUI2016IETO,MR3084175} propose a collocation method for solving general systems of linear integral-algebraic equations (IAEs).  Additionally, Lubich et al. propose convolution quadrature as an approach to approximate the integral term of the memory kernel \cite{MR2231714,EduardoCuesta2006Cqtd,LubichC.1988Cqad,MR0923707}. For further rapid algorithms, refer to \cite{KwonKiwoon2003Apmf,MR1163355,HuangYuxiang2022Eeot,MR3958333,GuoLing2019Emmf} and the references therein.

In this paper, we devise a novel rapid algorithm by incorporating the incremental Singular Value Decomposition (SVD) technique. We note that our innovative method keeps uniformity in spite of the form of  memory kernel $ K(t) $ via the incremental SVD technique. The incremental SVD is  initially proposed by Brand in \cite{brand2002incremental} as an efficient method for computing the SVD of a low-rank matrix. This approach begins by initializing the incremental SVD with a small dataset and subsequently updating it as new data becomes available.  Its applications span various domains, including recommender systems \cite{frolov2017tensor}, proper orthogonal decomposition \cite{MR4017489,MR3775096}, dynamic mode decomposition \cite{hemati2014dynamic}, and visual tracking \cite{ross2008incremental}.

The algorithm necessitates the computation of thousands or even millions of orthogonal matrices, which are subsequently multiplied together. Nonetheless, these multiplications have the potential to destroy the orthogonality. Consequently, many reorthogonalizations are imperative in practical implementation.  Brand addressed this matter in \cite{MR2214744}, noting, ``It is an open question how often this is necessary to guarantee a certain overall level of numerical precision; it does not change the overall complexity.'' A subsequent work \cite{zhang2022answer} provided a response to this query, suggesting a method to mitigate the need for extensive orthogonal matrix computations and thereby eliminating the necessity for reorthogonalizations. Moreover, they demonstrated that this modification does not adversely impact the algorithm's outcomes.

Our approach is to simultaneously solve the integro-differential equation \eqref{integro-diff-Eq} and compress the solution data  using the incremental SVD method. The  incremental SVD  algorithm can be easily used in conjunction with a time stepping code for simulating equation \eqref{integro-diff-Eq}. This approach enables storing solution data in several smaller matrices, alleviating the need for a  huge dense matrix as often seen in traditional methods.   Consequently, by presuming that the solution data demonstrates an approximate low-rank characteristic,  we are able to address the issue of data storage in solving the integro-differential equation \eqref{integro-diff-Eq}.

The subsequent sections of this paper are structured as follows for a coherent understanding. In  \Cref{sec2}, we provide a comprehensive overview of the standard finite element method used for spatial discretization, coupled with a second-order numerical scheme for time discretization to address the integro-differential equation \eqref{integro-diff-Eq} in cases where the kernel $K(t)$ remains non-singular. Our demonstration reveals a linear increase in storage requirements alongside quadratic growth in computational complexity with respect to the number of time steps $n$. Moving to  \Cref{sec3}, we delve into a detailed review of the enhanced incremental  SVD  algorithm introduced in \cite{zhang2022answer}. In  \Cref{ISVD_IDeq}, we introduce a novel algorithm to solve the integro-differential equation \eqref{integro-diff-Eq}, leveraging the improved incremental SVD methodology. Our novel approach ensures memory efficiency for low-rank data and exhibits linear growth in computational complexity relative to the progression of time steps. In   \Cref{Error_estimate}, we present a rigorous error analysis of our method when $K(t)$ remains non-singular, demonstrating that our approach attains error convergence comparable to conventional finite element methods. 
In   \Cref{singularkernel}, we extend our novel method to cases where $K(t)$ exhibits weak singularity, presenting corresponding results. The subsequent numerical experiments in  \Cref{Numerical_experiments} further validate the efficacy of our novel algorithm in terms of accuracy, memory utilization, and computational efficiency. {Moreover, we extend our innovative method to variable-order time-fractional equation in \Cref{Numerical_experiments}, the corresponding theoretical analysis will be given out in the future research. }Finally, we conclude by discussing potential avenues for future research in the conclusion.
\section{The finite element method for integro-differential euqations}\label{sec2}
In this section, our aim is to introduce the finite element method tailored for handling integro-differential equations \eqref{integro-diff-Eq} characterized by a non-singular kernel.

Let $\mathcal{T}_h$ represent a collection of regular simplices $K$ that partition the domain $\Omega$, and $\mathcal{P}_k(K)$ ($k\ge 1$) denote the polynomial space defined on the element $K$ with a maximum degree of $k$. Utilizing the triangulation $\mathcal{T}_h$, we can define the continuous piecewise finite element space $V_h$ as follows:
\begin{align*}
	V_h=\left\lbrace v_h\in H_0^1(\Omega): v_h|_K\in\mathcal{P}_{k}(K),\forall K\in \mathcal{T}_h \right\rbrace. 
\end{align*}

The semidiscrete Galerkin scheme for the integro-differential equation \eqref{integro-diff-Eq} can be expressed as follows:  seek $u_h(t)\in V_h$ such that
\begin{subequations}\label{semi-discretization}
	\begin{align}
		(u_{h,t},v_h)+\mathscr A(u_h,v_h)+\int_0^t  K(t-s)\mathscr {B}(u_h(s),v_h)\ {\rm d}s&=(f,v_h),\quad \forall v_h\in V_h,\\
		u_h(0)&=u_h^0\in V_h, 
	\end{align}
\end{subequations}
where $u_h^0$ corresponds to the projection of $u_0$ onto the space $V_h$,  $u_{h,t}$ denotes the time derivative of $u_h$, while $\mathscr{A}(\cdot,\cdot)$, $\mathscr{B}(\cdot,\cdot)$ represent the bilinear forms associated with the operator $\mathcal A$ and $ \mathcal B $, defined on $H_0^1(\Omega)\times H_0^1(\Omega)$, and take the following forms:
\begin{align*}
	\mathscr 	A(u,v)&=\sum_{i,j=1}^d \left(a_{ij}(x)\frac{\partial u}{\partial x_i},\frac{\partial v}{\partial x_j}\right)+(a(x)u,v),\\
	\mathscr 	B(u,v)&=\sum_{i,j=1}^d \left(b_{ij}(x)\frac{\partial u}{\partial x_i},\frac{\partial v}{\partial x_j}\right)+\sum_{i=1}^{d}(b_i(x)\frac{\partial u}{\partial x_i},v)+(b(x)u,v).
\end{align*}
Here, $(\cdot,\cdot)$ represents the inner product in $L^2(\Omega)$.

In the temporal discretization process, we utilize the second-order Crank-Nicolson scheme. To accomplish this, we define a set of $N+1$ nodes $\{t_i\}_{i=0}^N$ that partition the time domain $\left[0, T\right]$ into $N$ steps, where $0=t_0<t_1<t_2<\ldots< t_N=T$. Ensuring convergence, we assume a quasi-uniform partitioning of the time domain, indicating the existence of a positive constant $C_1>0$. This condition is expressed as:
\begin{align}\label{time-quasi-uniform}
	\max\limits_{1 \leq n \leq N}\Delta t_n=\Delta t\le C_1(\delta t)=C_1 \min_{1\le n\le N}\Delta t_n,
\end{align}
where $  \Delta t_j = t_j - t_{j-1}. $

For the temporal discretization of the integral term in $ \eqref{semi-discretization} $, we apply the midpoint rule, akin to the approach in \cite{GuoYingwen2022Ceaf}, to approximate the memory term. Consequently, the fully discrete scheme is represented as follows: given $u_h^0 \in V_h$, our objective is to determine $u_h^{n} \in V_h$ for all $n = 1, \ldots, N$, such that
\begin{align}\label{discrete-eq0}
	\left(\partial_t^+ u_h^n,v_h\right)+\mathscr A(\bar u_h^{n},v_h)+\mathscr{B}(\bar{u}_{K,h}^{\Delta t,n},v_h)=(\bar f^n,v_h),  \forall v_h\in V_h,
\end{align}
where
\begin{gather*}
	\bar{u}_h^n=\frac{u_h^n+u_h^{n-1}}{2},\quad 
	\bar{u}_{K,h}^{\Delta t,n}=(\bar{t}_n-t_{n-1})K(0)\bar{u}_h^n+\sum_{j=1}^{n-1}K(\bar{t}_n-\bar{t}_j)\bar{u}_h^j\Delta t_j,\\
	\bar{t}_n=\frac{t_n+t_{n-1}}{2},\; \bar{f}^n=\frac{f(t_{n-1})+f(t_n)}{2},\;
	\partial_t^+ u_h^n=\frac{u_h^n-u_h^{n-1}}{\Delta t_n}.
\end{gather*}
Next, we assume that  $
V_h=\textup{span}\left\lbrace \phi_1,\cdots,\phi_m\right\rbrace$,  and we define the matrices $M$, $A$, $ B $  and vector $b_{i+1} $ as follows:
\begin{gather}\label{eq_b1}
	\begin{split}
		M_{ij} =(\phi_i,\phi_j),  A_{ij}=\sum_{k,\ell=1}^{d}\left(a_{k\ell}(x)\frac{\partial \phi_i}{\partial x_k},\frac{\partial \phi_j}{\partial x_\ell}\right)+(a(x)\phi_i,\phi_j),  (b_{n})_j = (\bar f^{n},  \phi_j), \\
		B_{ij} =\sum_{k,\ell=1}^{d}\left(b_{k\ell}(x)\frac{\partial \phi_i}{\partial x_k},\frac{\partial \phi_j}{\partial x_\ell}\right)
		+\sum_{k=1}^{d}(b_k(x)\frac{\partial\phi_i}{\partial x_k},\phi_j)+(b(x)\phi_i,\phi_j), 
	\end{split}
\end{gather}
where  $X_{ij}$ represents the element of row $i$ and column $j$ of matrix $X$, and $(\alpha)_j$ denotes the $j$-th component of the vector $\alpha$. Let $u_{n}$ be the coefficients of $u_h^{n}$ at time $t_{n}$, given by:
\begin{align}\label{vevtor_demonstration}
	u_h^{n}=\sum_{j=1}^{m}(u_{n})_j\phi_j.
\end{align}

Substituting  \eqref{eq_b1} and \eqref{vevtor_demonstration} into \eqref{discrete-eq0}, we obtain the following algebraic system:
\begin{align}\label{algebraic_system}
	\begin{split}
		&(M+\frac{\Delta t_n}{2}A+\frac{\Delta t_n}{2}(\bar{t}_n-t_{n-1})K(0)B)u_n\\
		&\quad=(M-\frac{\Delta t_n}{2}A-\frac{\Delta t_n}{2}(\bar{t}_n-t_{n-1})K(0)B)u_{n-1}\\
		&\qquad-\Delta t_n\sum_{j=1}^{n-1}K(\bar{t}_n-\bar{t}_j)\Delta t_jB\bar{u}_{j}+\Delta t_nb_n.
	\end{split}
\end{align}
By solving the above algebraic system, we can obtain the solution to  \eqref{discrete-eq0}.
The algorithm is presented in \cref{alg101}. 
\begin{algorithm}
	\caption{(Finite element and backward Euler  method for solving  equation \eqref{integro-diff-Eq})}
	\label{alg101}
	\begin{algorithmic}[1]
		\REQUIRE  
		$ \Delta t$, $M\in\mathbb{R}^{m\times m}$, $A\in\mathbb{R}^{m\times m}$,$B\in\mathbb{R}^{m\times m}$, $u_0$
		\STATE Set $ {U}=\texttt{zeros}(m,N+1) $, $U(:,1) = u_0$.
		\FOR{$ n=1 $ to $ N $}
		\STATE Set $ \displaystyle A_1=M+\frac{\Delta t_n}{2}A+\frac{\Delta t_n}{2}(\bar{t}_n-t_{n-1})K(0)B,
		A_2=M-\frac{\Delta t_n}{2}A-\frac{\Delta t_n}{2}(\bar{t}_n-t_{n-1})K(0)B $
		\STATE Compute the coefficients $ K(\bar{t}_n-\bar{t}_j)$ and  get the load vector  $ b_{n}$.
		\STATE $ \displaystyle \widetilde b_{n}=A_2U(:,n)-\Delta t_nB\sum_{j=1}^{n-1}K(\bar{t}_n-\bar{t}_j)\Delta t_j\frac{U(:,j+1)+U(:j)}{2}+\Delta t_nb_n$.
		\STATE Solve $ A_1u_{n}=\widetilde b_{n}$
		\STATE $ {U}(:,n+1)=u_{n} $
		\ENDFOR
		\ENSURE $ u_n$  
	\end{algorithmic}
\end{algorithm}

To compute the numerical solution $u_h^{n+1}$, all the preceding time numerical solutions $\{u_h^j\}_{j=0}^{n}$ must be available.  Therefore, there is a potential need for us to store all the numerical solutions $\{u_h^j\}_{j=0}^{i}$.  We remark that not all memory kernel $ K(t) $ will result in this issue. For example, for  the kernel $ K(t)=\rho e^{-\delta t} $  or $ K(t)=\sin{(t)} $, we can avoid storing data by slightly modifing fully discretization scheme. In this work, we consider some special kernels, which needs to store data for the next step computation. Note that the storage cost of the history term is
\begin{align}\label{FEM_memory}
	\mathcal O(mn),
\end{align}
and the computational cost is
\begin{align}\label{FEM_computational}
	\sum_{i=1}^n \sum_{j=0}^{i} \mathcal O (m) = \mathcal O(mn^2).
\end{align}

Consequently, the storage requirement increases linearly while the computational complexity grows quadratically with the number of time steps $n$.  This poses a substantial challenge for data storage and numerical simulations.

Hence, we put forward an unified fast algorithm with the help of the incremental singular value decomposition (SVD) method to address the memory and computational complexity issue mentioned above, provided we make the assumption that the solution data exhibits approximate low rank.
\section{The incremental SVD method}\label{sec3}

We begin by introducing several crucial definitions and concepts that are essential for understanding the incremental SVD method. Given a vector $u \in \mathbb{R}^m$ and an integer $r$ satisfying $r \leq m$, the notation $u(1: r)$ denotes the first $r$ components of $u$. Likewise, for a matrix $U \in \mathbb{R}^{m \times n}$, we use the notation $U(p: q, r: s)$ to refer to the submatrix of $U$ that encompasses the entries from rows $p$ to $q$ and columns $r$ to $s$.

Throughout this section, we make the assumption that the rank of the matrix $U \in \mathbb{R}^{m \times n}$ is low, specifically denoted by $\texttt{rank} (U) \ll \min\{ m, n\}$.

Next, we present an improved version of Brand’s incremental SVD algorithm from \cite{zhang2022answer}. The algorithm updates the SVD of a matrix when one or more columns are added to the matrix.  We split the process into four steps.
\subsection*{Step 1: Initialization} 
Assuming that the first column of matrix $U$, denoted as $u_1$, is non-zero, we can proceed to initialize the SVD of $u_1$ using the following approach:
\begin{align*}
	\Sigma=(u_1^\top u_1)^{1/2},\qquad Q=u_1\Sigma^{-1},\qquad R=1.
\end{align*}
The algorithm is shown in  \Cref{alg201}.
\begin{algorithm}
	\caption{(Initialize ISVD)}
	\label{alg201}
	\begin{algorithmic}[1]	
		\REQUIRE  %算法的输入参数：Input
		$ u_1\in \mathbb{R}^m$\\
		\STATE $\Sigma=(u_1^\top u_1)^{1/2};\quad Q=u_1\Sigma^{-1};\quad R=1 $\\
		\ENSURE $ Q,\Sigma, R $ %算法的输出：Output
	\end{algorithmic}
\end{algorithm}

Assuming we already have  the  truncated SVD of rank $k$ for the first $\ell$ columns of matrix $U$, denoted as $U_{\ell}$:
\begin{align}\label{eq201}
	U_\ell \approx Q\Sigma R^\top,\quad \textup{with} \quad Q^\top Q=I_k,\quad R^\top R=I_k,\quad \Sigma= \texttt{diag}(\sigma_1,\cdots,\sigma_k),
\end{align}
where $ \Sigma \in \mathbb{R}^{k\times k}$ is a diagonal matrix with the $k$ ordered singular values of $ U_\ell$ on the diagonal, $ Q\in \mathbb{R}^{m\times k} $ is the matrix of the corresponding $k$ left singular vectors of $ U_\ell $ and $ R\in\mathbb{R}^{\ell\times k} $ is the matrix of the corresponding $ k $ right singular vectors of $ U_\ell $.

Given our assumption that the matrix $U$ is low rank, it is reasonable to expect that most of the columns of $U$ are either linearly dependent or nearly linearly dependent on the vectors in $Q \in \mathbb{R}^{m \times k}$. Without loss of generality, we assume that the next $s$ vectors, denoted as $\left\lbrace u_{\ell+1}, \ldots, u_{\ell+s}\right\rbrace$, their residuals are less than a specified tolerance when projected onto the subspace spanned by the columns of $Q$. However, the residual of $u_{\ell+s+1}$ is larger than the given tolerance. In other words,
\begin{subequations}
	\begin{align}
		|u_i-QQ^\top u_i  | &< \texttt{tol}, \quad i=\ell+1,\cdots, \ell+s,\label{lessthantol}\\
		|u_i-QQ^\top u_i   |  &\ge \texttt{tol}, \quad i=\ell+s+1.\label{largerthantol}
	\end{align}
\end{subequations}
The symbol $|\cdot|$ denotes the Euclidean norm within the realm of $\mathbb{R}^m$.
\subsection*{Step 2: Update the SVD of $ U_{\ell+s} $ ($p$-truncation)}
By  the  assumption \eqref{lessthantol}, we have 
\begin{align*}
	U_{\ell+s}&= \left[ U_\ell\mid u_{\ell+1}\mid\cdots\mid u_{\ell+s}\right] \\
	&\approx \left[ Q\Sigma R^\top \mid u_{\ell+1}\mid\cdots\mid u_{\ell+s}\right] \\
	&\approx \left[ Q\Sigma R^\top \mid QQ^\top u_{\ell+1}\mid\cdots\mid QQ^\top u_{\ell+s}\right] \\
	&=Q \underbrace{\left[ \Sigma\mid Q^\top u_{\ell+1}\mid\cdots\mid Q^\top u_{\ell+s}\right] }_{Y}\left[\begin{array}{cc}
		R & 0 \\
		0 & I_s
	\end{array}\right]^\top.  
\end{align*}    

We can obtain the truncated SVD of $U_{\ell+s}$ by computing the thin SVD of the matrix $Y$. Specifically, let $Y = Q_Y \Sigma_Y R_Y^\top$ be the SVD of $Y$, and split $ R_Y $ into $ \left[\begin{array}{cc}
	& R_Y^{(1)} \\
	& R_Y^{(2)}
\end{array}\right]$. With this, we can update the SVD of $U_{\ell+s}$ as follows:
\begin{align*}
	Q\leftarrow QQ_Y,\quad \Sigma\leftarrow\Sigma_Y,\quad  R\leftarrow \left[\begin{array}{cc}
		& RR_Y^{(1)} \\
		& R_Y^{(2)}
	\end{array}\right] \in \mathbb{R}^{(\ell+s)\times \ell}.
\end{align*}

It is worth noting that the dimensions of the matrices $Q$ and $\Sigma$ remain unchanged, and we need to incrementally store the matrix $W = \left[ Q^\top u_{\ell+1}\mid\cdots\mid Q^\top u_{\ell+s}\right]$. As $W$ belongs to $\mathbb{R}^{k\times s}$ where $k\leq r$ is relatively small, the storage cost for this matrix  is low.   
\subsection*{Step 3: Update the SVD of $ U_{\ell+s+1} $ (No truncation)}
Next, we proceed with the update of the SVD for $U_{\ell+s+1}$. Firstly, we compute the residual vector of $u_{\ell+s+1}$ by projecting it onto the subspace spanned by the columns of $Q$, i.e.,
\begin{align}\label{residual}
	e=u_{\ell+s+1}-QQ^\top  u_{\ell+s+1}.
\end{align}

First, we define $p = |e|$. Then, based on \eqref{largerthantol}, we deduce that $p > \textup{\texttt{tol}}$. Finally, we denote $\widetilde{e}$ as $e/p$.  With these definitions, we establish the following fundamental identity:
\begin{align*}
	U_{\ell+s+1}&= \left[U_{\ell+s}\mid u_{\ell+s+1} \right] \\
	&\approx \left[ Q\Sigma R^\top \mid p\widetilde{e}+QQ^\top u_{\ell+s+1}\right] \\
	&\approx [Q \mid \widetilde{e}] \underbrace{\left[\begin{array}{cc}
			\Sigma & Q^{\top} u_{\ell+s+1} \\
			0 & p
		\end{array}\right]}_{\bar{Y}}\left[\begin{array}{cc}
		R & 0 \\
		0 & 1
	\end{array}\right]^{\top},
\end{align*}

Let $ \bar{Q}\bar{\Sigma}\bar{R}^\top  $  be the full SVD of $\bar Y$. Then  the SVD of $ U_{\ell+s+1} $ can be approximated by
\begin{align*}
	U_{\ell+s+1}  \approx (\left[ Q\mid \widetilde{e}\right]\bar{Q} )\bar{\Sigma} \left(\left[\begin{array}{cc}
		R & 0 \\
		0 & 1
	\end{array}\right]\bar{R}\right)^\top.
\end{align*}

With this, we can update the SVD of $U_{\ell+s+1}$ as follows:
\begin{align*}
	Q\leftarrow (\left[ Q\mid \widetilde{e}\right])\bar{Q}, \quad \Sigma\leftarrow\bar{\Sigma}, \quad  R\leftarrow \left[\begin{array}{cc}
		R & 0 \\
		0 & 1
	\end{array}\right]\bar{R}.
\end{align*}
It is worth noting that, in this case, the dimensions of the matrices $Q$ and $\Sigma$ increase.
\begin{remark}
	Theoretically, the residual vector $e$ in \eqref{residual} is orthogonal to the vectors in  the subspace spanned by the columns of $Q$. However, in practice, this orthogonality can be completely lost, a fact that has been confirmed by numerous numerical experiments \cite{MR4017489,MR3775096,MR3594691}.  In \cite{MR2167744}, Giraud et al. stressed that exactly two iteration-steps are enough to keep the orthogonality. To reduce computational costs, Zhang \cite{zhang2022answer} suggested using the two iteration steps only when the inner product between $e$ and the first column of $Q$   exceeds a certain tolerance. Drawing from our experience, it is imperative to calibrate this tolerance to align closely with the machine error. For instance, as demonstrated in this paper, we consistently establish this tolerance as $10^{-14}$.
\end{remark}
\subsection*{Step 4: Singular value  truncation}
For many  PDE data sets, they may have a large number of nonzero singular values but most of them are very small. Considering the computational cost involved in retaining all of these singular values, it becomes necessary to perform singular value truncation. This involves discarding the last few singular values if they fall below a certain tolerance threshold.
\begin{lemma} \textup{\cite[Lemma 5.1]{zhang2022answer}}
	Assume that $\Sigma=\operatorname{diag}\left(\sigma_1, \sigma_2, \ldots, \sigma_k\right)$ with $\sigma_1 \geq \sigma_2 \geq \ldots \geq \sigma_k$, and $\bar{\Sigma}=\operatorname{diag}\left(\mu_1, \mu_2, \ldots, \mu_{k+1}\right)$ with $\mu_1 \geq \mu_2 \geq \ldots \geq \mu_{k+1}$. Then we have
	\begin{align}
		\label{eq302}
		\mu_{k+1} &\le p, \\
		\label{eq303}
		\mu_{k+1} &\le \sigma_k \leq \mu_k \leq \sigma_{k-1} \leq \ldots \leq \sigma_1 \leq \mu_1.
	\end{align}
	
\end{lemma}

The inequality $\eqref{eq302}$ indicates that, regardless of the magnitude of $p$, the last singular value of $\bar{Y}$ can potentially be very small. This implies that the tolerance set for $p$ cannot prevent the algorithm from computing exceedingly small singular values. Consequently, an additional truncation is necessary when the data contains numerous very small singular values. Fortunately, inequality $\eqref{eq303}$ assures us that only the last singular value of $\bar{Y}$ has the possibility of being less than the tolerance. Therefore, it suffices to examine only the last singular value.
\begin{itemize}
	
	\item[ (i)] If $\Sigma_Y(k+1, k+1) \geq \mathtt{tol} $, then
	$$
	Q \longleftarrow[Q \mid \widetilde{e}] Q_Y, \quad \Sigma \longleftarrow \Sigma_Y, \quad R \longleftarrow\left[\begin{array}{cc}
		R & 0 \\
		0 & 1
	\end{array}\right] R_Y.
	$$
	
	\item[(ii)] If $\Sigma_Y(k+1, k+1)<\mathtt{tol}$, then
	$$
	Q \longleftarrow[Q \mid \widetilde{e}] Q_Y(:, 1: k),  \quad \Sigma \longleftarrow \Sigma_Y(1: k, 1: k), \quad R \longleftarrow\left[\begin{array}{cc}
		R & 0 \\
		0 & 1
	\end{array}\right] R_Y(:, 1: k).
	$$
\end{itemize}
It is essential to note that $p$-truncation and no-truncation  do not alter the previous data, whereas singular value truncation may potentially change the entire previous data. However, we can establish the following bound:
\begin{lemma}\label{error_SingularValueTruncation}
	Suppose $Q\Sigma R^\top$ to be the SVD of $A\in \mathbb R^{m\times n}$, where $\{\sigma_i\}_{i=1}^r$ are the positive singular values. Let $B = Q(:,1:r-1)\Sigma(1:r-1,1:r-1)(R(:,1:r-1))^\top$. We have:
	\begin{align*}
		\max\{|a_1 - b_1|, |a_2- b_2|, \ldots, |a_n-b_n|\} \le \sigma_r.
	\end{align*}
	Here, $a_i$ and $b_i$ correspond to the $i$-th columns present in matrices $A$ and $B$ respectively. The symbol $|\cdot|$ denotes the Euclidean norm within the realm of $\mathbb{R}^m$.
\end{lemma}

The proof of \Cref{error_SingularValueTruncation} is straightforward, and thus we omit it here.  Moving forward, we will provide a summary of the aforementioned four steps in \Cref{alg202}.

\begin{algorithm}[h]
	\caption{(Update ISVD)}
	\label{alg202}
	\begin{algorithmic}[1]
		\REQUIRE  
		$ Q\in \mathbb{R}^{m\times k},\Sigma \in \mathbb{R}^{k\times k},R\in \mathbb{R}^{\ell\times k},\mathtt{tol}, W, Q_0,q,u_{\ell+1},$\\
		\STATE Set $ d=Q^\top u_{\ell+1};e=u_{\ell+1}-Qd;p=(e^\top e)^{1/2} ;$\\
		\IF{$p \ge \mathtt{tol}$}
		\IF{$q>0$}
		\STATE Set $ Y=\left[\Sigma\mid \mathtt{cell2mat}(W) \right]  $;  $ \left[Q_Y,\Sigma_Y,R_Y \right] =\mathtt{svd}(Y,{'\mathtt{econ}'}) $;
		\STATE Set $ Q_0=Q_0Q_Y,\Sigma=\Sigma_Y,R_1=R_{Y(1:k,:)},R_2=R_{Y(k+1:\mathtt{end},:)},R=\left[ \begin{array}{cc}
			RR_1\\
			R_2
		\end{array}\right]  $;  $ d=Q_Y^\top d; $
		\ENDIF
		\STATE Set $ Y=\left[ \begin{array}{cc}
			\Sigma & d\\
			0  &  p
		\end{array}\right]  $;
		$ \left[ Q_Y,\Sigma_Y,R_Y\right] =\mathtt{svd} (Y) $;   $ e=e/p; $
		\IF{$ \left| e^\top Q{(:,1)}\right|> 10^{-14}$ }
		\STATE $ e=e-Q(Q^\top e);p_1=(e^\top e)^{1/2};e=e/p_1 $;
		\ENDIF
		\STATE Set $ Q_0=\left[\begin{array}{cc}
			Q_0 &0\\
			0   &1
		\end{array} \right]Q_Y  $;
		
		\IF{$ \Sigma_Y(k+1, k+1)\ge \textup{\texttt{tol}}  $}  
		\STATE $Q=[Q\mid e]Q_0,\quad \Sigma=\Sigma_Y,\quad R=\left[\begin{array}{cc}
			R &0\\
			0   &1
		\end{array} \right]R_Y,\quad Q_0=I_{k+1};$  
		\ELSE
		\STATE
		$Q=[Q\mid e]Q_0(:,1:k),\Sigma=\Sigma_Y(1:k,1:k), R=\left[\begin{array}{cc}
			R &0\\
			0   &1
		\end{array} \right]R_Y(:,1:k),\quad Q_0=I_{k};$ 
		\ENDIF
		
		\STATE $ W=\left[ \right] ;q=0; $
		\ELSE
		\STATE $ q=q+1 $; $ W\left\lbrace q\right\rbrace =d $;
		\ENDIF
		\ENSURE $ Q,\Sigma, R,W,Q_0,q.$
	\end{algorithmic}
\end{algorithm}

\section{Incremental SVD method  for the integro-differential equation}\label{ISVD_IDeq}
This section focuses on the application of the incremental SVD algorithm to the  equation \eqref{integro-diff-Eq}.  Throughout this section,  we make the assumption that the solution data of equation  \eqref{integro-diff-Eq} exhibits approximate low rank. 

Our approach is to simultaneously solve the integro-differential equation and incrementally update the SVD of the solution. By doing so, we store the solutions at all time steps in the four matrices of the SVD. As a result, we are able to address the issue of data storage in solving the integro-differential equation \eqref{integro-diff-Eq}. 

Due to the errors that may arise during the data compression process and the potential alterations caused by singular value truncation to previous storage, it becomes necessary for us to modify the traditional scheme \eqref{discrete-eq0}. Below, we provide a brief discussion of our approach.

\begin{itemize}
	\item[(1)] Use the initial condition $u_h^0$ to compute the numerical solution at time step 1, which follows the traditional approach. However, we use $\widehat u_h^1$ to denote the numerical solution for consistency. Once we obtain $\widehat u_h^1$, we compress $\{u_h^0, \widehat u_h^1\}$ and denote the corresponding compressed data as $\{\widetilde u_h^{1,0}, \widetilde u_h^{1,1}\}$.
	
	\item[(2)] Use the compressed data $\{\widetilde u_h^{1,0}, \widetilde u_h^{1,1}\}$ to compute the numerical solution at time step 2, denoted by $\widehat u_h^2$. Once we obtain $\widehat u_h^2$, we compress $\{\widetilde u_h^{1,0}, \widetilde u_h^{1,1}\}$ and $\widehat u_h^2$, and denote the corresponding compressed data as $\{\widetilde u_h^{2,0}, \widetilde u_h^{2,1}, \widetilde u_h^{2,2}\}$.
	
	\item[(3)] At time step $i$, given the compressed data $\{\widetilde u_h^{i,j}\}_{j=0}^i$, we compute the numerical solution at time step $i+1$, denoted by $\widehat u_h^{i+1}$. We then compress $\{\widetilde u_h^{i,j}\}_{j=0}^i$ and $\widehat u_h^{i+1}$, and denote the corresponding compressed data as $\{\widetilde u_h^{i+1,j}\}_{j=0}^{i+1}$.
	
	\item[(4)] Continue the above process until we reach the final time step.
	
\end{itemize}

In summary, we apply our novel approach by incrementally compressing data at each time step to compute the numerical solutions throughout the process.

Based on the preceding discussion, we can present our formulation below, where we seek $\widehat u_h^{n} \in V_h$ that satisfies the following equation:
\begin{align}\label{ISVD_eq1}
	\left(\partial_t^+\widehat{u}_h^n,v_h\right)+\mathscr A(\bar{\widehat{u}}_h^{n},v_h)+\mathscr{B}(M^n,v_h)=(\bar f^{n},v_h) \quad \forall v_h\in V_h,
\end{align}
where
\begin{align*}
	M^n=\sum_{j=1}^{n-1}\Delta t_jK(\bar{t}_n-\bar{t}_j)\bar{\widetilde{u}}_h^{n-1,j}+\frac{\Delta t_n}{2}K(0)\bar{\widehat{u}}_h^n,\\
	\bar{\widetilde{u}}_h^{n-1,j}=\frac{\widetilde{u}_h^{n-1,j}+\widetilde{u}_h^{n-1,j-1}}{2},\qquad \bar{\widehat{u}}_h^{n}=\frac{\widehat{u}_h^n+\widehat{u}_h^{n-1}}{2}.
\end{align*}

Subsequently, we express equation \eqref{ISVD_eq1} into matrix form to highlight the benefits of utilizing the incremental SVD for solving the integro-differential equation more distinctly.  To do this, let $\widehat u_{i+1}$ and $\widetilde u_{i,j}$ denote the coefficient of $\widehat u_h^{i+1}$ and $\widetilde u_h^{i,j}$, respectively. Then we seek  a solution $\widehat u_{i+1}\in \mathbb R^m$ that satisfies the following equation:
\begin{align}\label{ISVD_eq2}
	\begin{split}
		&(M+\frac{\Delta t_n}{2}A+\frac{\Delta t_n}{2}(\bar{t}_n-t_{n-1})K(0)B)\widehat{u}_n\\
		&\quad=(M-\frac{\Delta t_n}{2}A-\frac{\Delta t_n}{2}(\bar{t}_n-t_{n-1})K(0)B)\widehat{u}_{n-1}\\
		&\qquad-\Delta t_nB\sum_{j=1}^{n-1}K(\bar{t}_n-\bar{t}_j)\Delta t_j\frac{\widetilde{u}_{n-1,j}+\widetilde{u}_{n-1,j-1}}{2}+\Delta t_nb_n.
	\end{split}
\end{align}
Here, $\{\widetilde u_{i,j}\}_{j=0}^{i}$  represents the data that has been compressed from $\{\widetilde u_{i-1, 0}, \ldots, \\ \widetilde u_{i-1, i-1}, \widehat u_i\}$ using the incremental SVD algorithm. We assume that $Q_i$, $\Sigma_i$, $R_i$, and $W_i$ are the matrices associated with this compression process. In other words,
\begin{align}\label{ISVD_eq3}
	[\widetilde u_{i-1,0}\mid \cdots\mid \widetilde u_{i-1,i-1}\mid \widehat u_i]  \stackrel{\text{ Compress }}{ \longrightarrow }  Q_i [\Sigma_iR_{i}\mid W_{i}]^\top = 	[\widetilde u_{i,0}\mid \cdots\mid \widetilde u_{i,i}].
\end{align}
Let $[\Sigma_iR_{i}\mid W_{i}]^\top$ be denoted as $X_i$.  Accordingly, equation \eqref{ISVD_eq2} can be rewritten as follows:
\begin{align}\label{ISVD_eq4}
	\begin{split}
		&(M+\frac{\Delta t_n}{2}A+\frac{\Delta t_n}{2}(\bar{t}_n-t_{n-1})K(0)B)\widehat{u}_n\\
		&\quad=(M-\frac{\Delta t_n}{2}A-\frac{\Delta t_n}{2}(\bar{t}_n-t_{n-1})K(0)B)\widehat{u}_{n-1}\\
		&\qquad-\Delta t_nBQ_i\sum_{j=1}^{n-1}K(\bar{t}_n-\bar{t}_j)\Delta t_j\frac{X(:,j+1)+X(:,j)}{2}+\Delta t_nb_n.\textbf{}
	\end{split}
\end{align}

Once $\widehat u_{i+1}$ is obtained, we update the SVD of $[\widetilde u_{i,0}\mid \cdots\mid \widetilde u_{i,i}\mid \widehat u_{i+1}]$ using $Q_{i}$, $\Sigma_{i}$, $R_{i}$, $W_{i}$, and $\widehat u_{i+1}$ based on the principles of the incremental SVD method. This update process is illustrated  in  \Cref{ISVD_figure}. The corresponding algorithm is shown in \cref{alg301}.
\begin{figure}[tbh]
	\centerline{
		\hbox{\includegraphics[width=4.3in]{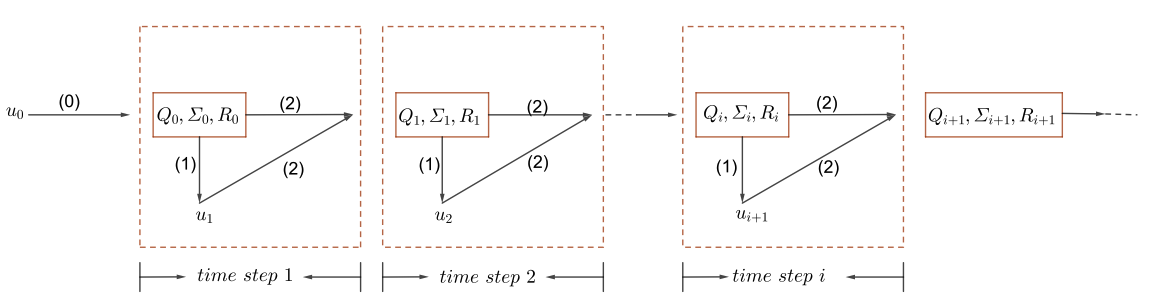}}}
	\caption{The process of using the incremental SVD to solve the integro-differential equation.}
	\label{ISVD_figure}
	\centering
\end{figure}	
Throughout the remainder of this section, we will examine the memory and computational cost pertaining to the history term in our novel approach. Our data storage involves four matrices: $Q_i$, $\Sigma_i$, $R_i$, and $W_i$, resulting in a memory cost of $\mathcal O((m+n)r)$, where $r$ represents the rank of the solution data. By taking into account our assumption that $r\ll \min\{m, n\}$, we can compare this memory cost with that of the traditional approach presented in \eqref{FEM_memory}, which illuminates a noteworthy reduction in our innovative method.

Moving on to the computational cost, which also encompasses the cost of the incremental SVD, it can be summarized as follows:
\begin{align*}
	\mathcal O(mnr) + \sum_{i=1}^n \sum_{j=0}^{i} \mathcal O(r) = \mathcal O(mnr + rn^2).
\end{align*}

Here, once again, $r$ represents the rank of the solution data. Based on our assumption that $r\ll \min\{m, n\}$, we can compare the computational cost in \eqref{FEM_computational} to that of the traditional approach, revealing that our approach experiences only linear growth, rather than quadratic growth as observed in the traditional approach.
\begin{algorithm}[h]
	\caption{(Incremental SVD for integro-differential equation)}
	\label{alg301}
	\begin{algorithmic}[1]
		\REQUIRE  $ \Delta t$,$N$, $M\in \mathbb{R}^{m\times m}$, $A\in \mathbb{R}^{m\times m}$,$B\in \mathbb{R}^{m\times m}$, $u_0$, $\texttt{tol},$\\
		\STATE Set $ W=\left[ \right]$; $Q_0=1$ $  \widehat u=u_0 $ ;\\
		\STATE $ \left[ Q,\Sigma,R\right] =\mathtt{InitializeISVD}(u_0) $; \hspace{6.1cm} \Cref{alg201}
		\FOR{$ n=1,\cdots,N$}
		\STATE $ A_1=M+\frac{\Delta t_n}{2}A +\frac{\Delta t_n}{2}(\bar{t_n}-t_{n-1})K(0)B$;
		\STATE $ A_2=M-\frac{\Delta t_n}{2}A -\frac{\Delta t_n}{2}(\bar{t_n}-t_{n-1})K(0)B $;
		\STATE Compute the coefficients $ K(\bar{t}_n-\bar{t}_j)$, $j=1,\ldots, n-1$,  and get the load vector   $ b_{n} $;
		\IF{$W$  is empty}
		\STATE $ \mathtt{MQ}=MQ$; $X=\Sigma R^\top  $;$ \mathtt{BQ}=BQ $;
		\STATE $\widetilde  b_{n}=A_2 \widehat u-\Delta t_n  \mathtt{BQ}  \sum_{j=1}^{n-1}\Delta t_jK(\bar{t}_n-\bar{t}_j)\frac{X(:,j)+X(:,j+1)}{2}+\Delta t_nb_{n} $;
		\STATE Solve $ A_1  \widehat u=\widetilde  b_{n} $;
		\STATE $ \left[Q,\Sigma,R,W,Q_0,q \right]=\mathtt{UpdateISVD}(Q,\Sigma,R,\mathtt{tol},W,Q_0,q,  \widehat u) $; \hspace{1cm} \Cref{alg202}
		\IF{$W$ is not empty}
		\STATE $ X=\Sigma R^\top  $;
		\ENDIF
		\ELSE
		\STATE  $ D=\left[X \quad \mathtt{cell2mat}(W) \right]  $;
		\STATE $\widetilde  b_{n}=A_2 \widehat u-\Delta t_n  \mathtt{BQ}  \sum_{j=1}^{n-1}\Delta t_jK(\bar{t}_n-\bar{t}_j)\frac{D(:,j)+D(:,j+1)}{2}+\Delta t_nb_{n} $;
		\STATE Solve $ \widetilde{A} \widehat u=\widetilde b_{i+1}$;
		\STATE $ \left[ Q,\Sigma,R,W,Q_0,q\right] =\mathtt{UpdateISVD}(Q,\Sigma,R,\mathtt{tol},W,Q_0,q, \widehat u) $; \hspace{1cm} \Cref{alg202}
		\ENDIF
		\ENDFOR
		\ENSURE $\widehat u.$
	\end{algorithmic}
\end{algorithm}

\section{Error estimate}\label{Error_estimate}
In this section, we derive the error between the solution of the scheme \eqref{ISVD_eq1} and the exact solution that satisfies the equation \eqref{integro-diff-Eq}.
\subsection{Assumptions and Main Result} We assume throughout that  $\Omega$ is a bounded convex polyhedral domain,    the data  of  \eqref{integro-diff-Eq}  satisfies the following condition:  
\begin{assumption}\label{A2} 
	Let $\|\psi\|_a^2= \mathscr A(\psi, \psi)$ for $\psi \in H_0^1(\Omega)$,  $\displaystyle K_0=\int_0^T K(t) {\ \rm d}t$ and $ K(t)\in H^2 [0,T]  $, the following inequality holds:
	\begin{align*}
		c_0 K_0 <1,
	\end{align*}
	where the constant $ c_0>0 $ satisfies
	\begin{align*}
		\left| \mathscr {B}(u,v)\right|\le c_0\left\|u \right\|_a \left\|v\right\|_a,\quad \forall  u,v\in H_0^1(\Omega). 
	\end{align*}
\end{assumption}

\begin{remark}
	We emphasize that Assumption \ref{A2} ensures the dominance of the operator $\mathscr A$ over the integral term. However, as elucidated further in the subsequent sections, our requirements necessitate solely $ K(t)\in H^2[0,T] $ instead of assuming Assumption \ref{A2}. This modified requirement holds if the kernel $ K(t) $ remains positive definite, and $\mathscr{B}$ adheres to nonnegative symmetric properties in \Cref{singularkernel}. Our assumption, in this context, proves to be more encompassing and simpler to verify in comparison to the assumptions outlined in \cite{MR1686149}, where the prerequisites are more stringent to achieve a sharp decay rate.
\end{remark}

Now, we state the main result of our paper.
\begin{theorem}\label{main_res}
	Let $u$ and $\widehat {u}_h^n$ represent the solutions of \eqref{integro-diff-Eq} and \eqref{ISVD_eq1}, respectively. Throughout the entire process of the incremental SVD algorithm, the tolerance \textup{\texttt{tol}} is uniformly applied to both $p$-truncation and singular value truncation. Assuming the validity of Assumption \ref{A2}, if $u \in H^2(0,T;H^{k+1}(\Omega)) $, then the subsequent error bound holds:
	\begin{align}\label{error_est}
		\left\| u(T)-\widehat {u}_h^N\right\|\le C(h^{k+1}+\Delta t^2)+( {T_{sv}}+1)\sqrt{T(1+\gamma^{-1})\sigma(A)}\mathtt{tol},
	\end{align}
	where $C$ and $ \gamma \in (0,2c_0^{-1}K_0^{-1}-2)  $ are two positive constants, independent of $h$, $\Delta t$, and \textup{\texttt{tol}}, and $\sigma(A)$ represents the spectral radius of the stiffness matrix $A$, $ T_{sv}$ signifies  the total number of times singular value truncation is applied.
\end{theorem}
\begin{remark}
	In \Cref{main_res}, if $ K(t)\in H^2[0, T]$  is positive definite and $ \mathcal{B} $ is a symmetric definite elliptic operator, without  Assumption \ref{A2}, \eqref{error_est} still holds.
\end{remark}
\subsection{Proof of \Cref{main_res}}
We begin by giving an error bound between the solution of the standard finite element method given by equation \eqref{discrete-eq0} and the solution of the Non-Fickian model \eqref{integro-diff-Eq}. Additionally, we derive an error bound between the solution of the standard finite element method \eqref{discrete-eq0} and our novel scheme \eqref{ISVD_eq1}. By applying the triangle inequality, we obtain a straightforward error bound between the solution of the Non-Fickian model \eqref{integro-diff-Eq} and our novel scheme \eqref{ISVD_eq1}.

\begin{lemma}\label{fem_error}
	Let $ u $ and $ u_h $ represent the solutions of $ \eqref{integro-diff-Eq} $ and $ \eqref{semi-discretization} $, respectively. Assuming the validity of  Assumption \ref{A2}, alongside $ \| u_0-u_h^0\|\le Ch^{k+1}\left\| u_0\right\|_{k+1} $, if   \eqref{time-quasi-uniform}   holds and $u \in H^2(0,T;H^{k+1}(\Omega))$, then there exists a constant $ C>0 $, independent of $ h $ and $ \displaystyle \Delta t=\max_{1 \leq n \leq N}\Delta t_n $, such that:
	\begin{align}\label{eq4001}
		\left\| u(t_n)-u_h^n\right\|&\le C(h^{k+1}+\Delta t^2).
	\end{align}
\end{lemma}

Following this, we will provide a concise proof for \Cref{fem_error}. To facilitate this, we will first present some fundamental formulas crucial for subsequent error estimations, employing integration by parts extensively.
\begin{lemma}\textup{\cite[Lemma 4.2]{GuoYingwen2022Ceaf}}\label{integration_by_part}
	Employing integration by parts, it follows that for all $\varphi \in H^2(0,T)$, the following relations hold:
	\begin{align*}
		& \Delta t_n\bar{\varphi}\left(t_n\right)- \int_{t_{n-1}}^{t_n} \varphi(t) d t=\frac{1}{2 } \int_{t_{n-1}}^{t_n}\left(t-t_{n-1}\right)\left(t_n-t\right) \varphi_{t t}(t) d t, \\
		& \bar{\varphi}\left(t_n\right)-\varphi\left(\bar{t}_n\right)=\frac{1}{2} \int_{t_{n-1}}^{\bar{t}_n}\left(t-t_{n-1}\right) \varphi_{t t} d t+\frac{1}{2} \int_{\bar{t}_n}^{t_n}\left(t_n-t\right) \varphi_{t t} d t, \\
		& \frac{\Delta t_n\varphi\left(\bar{t}_n\right)}{2}- \int_{t_{n-1}}^{\bar t_n} \varphi(t) d t= \int_{t_{n-1}}^{\bar t_n}\left(t-t_{n-1}\right) \varphi_t(t) d t.
	\end{align*}
	
\end{lemma}

To prove  \Cref{fem_error}, we introduce the well-known Ritz-Volterra projection $ R_hu $ of the solution $u$ defined  as follows:
\begin{align*}
	\mathscr{A}(u-R_hu,v_h)+\int_0^tK(t-s)\mathscr{B}(u(s)-R_hu(s),v_h)\ {\rm d}s=0,\ \forall v_h\in V_h.
\end{align*}
The following lemma shows an approximation result   .
\begin{lemma}\textup{\cite[Lemma 3.3]{MR1686149}}\label{lemma505}
	Under the   Assumption \ref{A2}, there exists a constant $C_0>0$, independent of $h$ and time $t$ such that for $\rho=u-R_h u$,
	\begin{align*}
		\|\rho_t(t)\| \leq C_0 h^{k+1}\{|\!|\!| u_t(t)|\!|\!|_{k+1,R}+(R(t)+\int_0^tR(t-s)R(s)\ {\rm d}s)\|u_0\|_{k+1}\},
	\end{align*}
	where
	\begin{align*}
		|\!|\!| u_t|\!|\!|_{k+1, R}= & \|u_t(t)\|_{k+1}+\int_0^t R(t-s)\|u_t(s)\|_{k+1} \ \rm{d} s\\ 
		& +\int_0^t R(t-s)\left(\int_0^s R(s-\tau)\|u_t(\tau)\|_{k+1} \ \rm{d} \tau\right) \ \rm{d} s,
	\end{align*}
	and $R(t)$ is the resolvent of $K(t)$ and is defined by
	\begin{align*}
		R(t)=K(t)+\int_0^t K(t-s) R(s) \ {\rm d} s, \quad t \geq 0.
	\end{align*}
\end{lemma}
For the sake of simplicity in notation, we will henceforth use the following symbols: for all $ n=1,2,\cdots,N $,
\begin{align*}
	&\bar{t}_n=\frac{t_{n-1}+t_{n}}{2},\qquad u^n=u(t_n),\ \bar{u}^n=\frac{u^n+u^{n-1}}{2},\\
	&F(t)=\int_0^tK(t-s)u(s)\ {\rm d}s,\qquad  e_n=u^n-u_h^n,\\
	&\rho_n=u^n-R_hu^n,\qquad \theta_n=R_hu^n-u_h^n,\\
	&\bar{e}_{K}^{\Delta t,n}=\sum_{j=1}^{n-1}\Delta t_jK(\bar{t}_n-\bar{t}_j)\bar{e}_j+\frac{\Delta t_n}{2}K(0)\bar{e}_n,\\
	&\bar{\rho}_{K}^{\Delta t,n}=\sum_{j=1}^{n-1}\Delta t_jK(\bar{t}_n-\bar{t}_j)\bar{\rho}_j+\frac{\Delta t_n}{2}K(0)\bar{\rho}_n,\\
	&\bar{\theta}_{K}^{\Delta t,n}=\sum_{j=1}^{n-1}\Delta t_jK(\bar{t}_n-\bar{t}_j)\bar{\theta}_j+\frac{\Delta t_n}{2}K(0)\bar{\theta}_n.
\end{align*}
Next, we give a proof of \Cref{fem_error}.
\begin{proof}
	Since $ u(t) $ and $ u_h^n $ satisfy the following equations:
	\begin{align}
		\label{varational-formula}
		(u_t,v_h)+\mathscr A(u,v_h)+\int_0^t K(t-s) \mathscr B(u(s),v_h)\ {\rm d}s&=(f(x,t),v_h),\\
		\label{fully-dis-eq}
		\left(\frac{u_h^{n}-u_h^{n-1}}{\Delta t_n},v_h\right)+\mathscr A(\bar u_h^{n},v_h)+\mathscr{B}(\bar{u}_{K,h}^{\Delta t,n},v_h)&=(\bar f^n,v_h).
	\end{align}
	We integrate $ \eqref{varational-formula} $ from $ [t_{n-1},t_n] $ and subtract $ \eqref{fully-dis-eq} $ from $ \eqref{varational-formula} $ to obtain 
	\begin{align*}
		&(\frac{e_n-e_{n-1}}{\Delta t_n},v_h)+\mathscr A(\bar{e}_n,v_h)+\mathscr B(\bar{e}_{K}^{\Delta t,n},v_h)\\
		&=\frac{1}{\Delta t_n}\int_{t_{n-1}}^{t_n}(f(t),v_h) \ {\rm d} t-(\bar{f}^n,v_h)
		+\mathscr A(\bar{u}^n,v_h)\\
		&\quad -\frac{1}{\Delta t_n}\int_{t_{n-1}}^{t_n}\mathscr A(u(t),v_h) \ {\rm d} t +\mathscr B(\bar{u}_K^{\Delta t,n},v_h)\\
		&\quad -\frac{1}{\Delta t_n}\int_{t_{n-1}}^{t_n}\int_{0}^{t}K(t-s)\mathscr B(u(s),v_h)\ {\rm d}s{\rm d}t.
	\end{align*}
	Observing that $ e_n=\rho_n+\theta_n $, we deduce that
	\begin{align}\label{error-eq}
		\begin{split}
			&(\frac{\theta_n-\theta_{n-1}}{\Delta t_n},v_h)+\mathscr A(\bar{\theta}_n,v_h)+\mathscr B(\bar{\theta}_{K}^{\Delta t,n},v_h)\\
			&=\frac{1}{\Delta t_n}\int_{t_{n-1}}^{t_n}(f(t),v_h)\ {\rm d}t-(\bar{f}^n,v_h)
			+\mathscr A(\bar{u}^n,v_h)\\
			&\quad -\frac{1}{\Delta t_n}\int_{t_{n-1}}^{t_n}\mathscr A(u(t),v_h)\ {\rm d}t  +\mathscr B(\bar{u}_K^{\Delta t,n},v_h)\\
			&\quad -\frac{1}{\Delta t_n}\int_{t_{n-1}}^{t_n}\int_{0}^{t}K(t-s)\mathscr B(u(s),v_h)\ {\rm d}s{\rm d}t\\
			&\quad -(\frac{\rho_n-\rho_{n-1}}{\Delta t_n},v_h)-A(\bar{\rho_n},v_h)-B(\bar{\rho}_{K}^{\Delta t,n},v_h).
		\end{split}
	\end{align}
	By substituting $ v_h=2\Delta t_n\bar{\theta}_n $ into $ \eqref{error-eq} $, the equation simplifies to:
	\begin{align}
		\begin{split}\label{fem_proof_eq}
			&\|\theta_n\|^2-\|\theta_{n-1}\|^2+2\Delta t_n\|\bar{\theta}_n\|_a^2+2\Delta t_n\mathscr{B}(\bar{\theta}_{K}^{\Delta t,n},\bar{\theta}_n)\\
			&=2\int_{t_{n-1}}^{t_n}(f(t),\bar{\theta}_n)\ {\rm d}t-2\Delta t_n(\bar{f}^n,\bar{\theta}_n)
			+2\Delta t_n\mathscr A(\bar{u}^n,\bar{\theta}_n)\\
			&\quad -2\int_{t_{n-1}}^{t_n}\mathscr A(u(t),\bar{\theta}_n)\ {\rm d}t 
			+2\Delta t_n\mathscr B(\bar{u}_K^{\Delta t,n},\bar{\theta}_n)\\
			&\quad -2\int_{t_{n-1}}^{t_n}\int_{0}^{t}K(t-s)\mathscr B(u(s),\bar{\theta}_n)\ {\rm d}s{\rm d}t-2\Delta t_n(\frac{\rho_n-\rho_{n-1}}{\Delta t_n},\bar{\theta}_n)\\
			&\quad -2\Delta t_n\mathscr A(\bar{\rho_n},\bar{\theta}_n)-2\Delta t_n\mathscr B(\bar{\rho}_{K}^{\Delta t,n},\bar{\theta}_n)\\
			&=\sum_{i=1}^{9}R_i.
		\end{split}
	\end{align}
	Next, we turn to estimate the terms $ R_1-R_9$.  Employing Young's inequality and \Cref{integration_by_part}, we utilize these techniques specifically for $R_1-R_4$ and $ R_7 $, resulting in:
	\begin{align*}
		R_1+R_2&=-\int_{t_{n-1}}^{t_n}(t-t_{n-1})(t_n-t)(f_{tt}(t),\bar{\theta}_n)\ {\rm d}t\\
		&\le C\Delta t^{\frac{5}{2}}(\int_{t_{n-1}}^{t_n}\|f_{tt}\|^2\ {\rm d}t)^{\frac{1}{2}}\|\bar{\theta}_n\|\\
		&\le C\Delta t^4\int_{t_{n-1}}^{t_n}\|f_{tt}\|^2\ {\rm d}t+\frac{\mu\Delta t}{5}\|\bar{\theta}_n\|_a^2,\\
		R_3+R_4&=\mathscr{A}(\int_{t_{n-1}}^{t_n}(t-t_{n-1})(t_n-t)u_{tt}\ {\rm d}t,\bar{\theta_n})\\
		&\le C\Delta t^{\frac{5}{2}}(\int_{t_{n-1}}^{t_n}\|u_{tt}\|_a^2\ {\rm d}t)^{\frac{1}{2}}\|\bar{\theta}_n\|_a\\
		&\le C\Delta t^4\int_{t_{n-1}}^{t_n}\|u_{tt}\|_a^2\ {\rm d}t+\frac{\mu\Delta t}{5}\|\bar\theta_n\|_a^2,\\
		R_7&=2\Delta t_n(\partial_t^+\rho_n,\bar{\theta}_n)\\
		&\le C\Delta t_n\|\partial_t^+\rho_n\|^2+\frac{\mu\Delta t}{5}\|\bar{\theta}_n\|_a^2.
	\end{align*}	
	Here, $ \mu\in (0,1) $ represents a constant that will be explicitly defined later. Regarding $ R_5+R_6 $, we rephrase and express them as follows:
	\begin{align*}
		R_5+R_6&=2\Delta t_n\mathscr{B}(\bar u_{K}^{\Delta t,n}-\frac{1}{\Delta t_n}\int_{t_{n-1}}^{t_n}F(t)\ {\rm d}t,\bar{\theta}_n)\\
		&=2\Delta t_n\sum_{j=1}^{n-1}K(\bar{t}_n-\bar{t}_j)\Delta t_j\mathscr{B}(\bar{u}_j-u(\bar{t}_j),\bar{\theta}_n)
		\\
		&\quad+\Delta t_n^2K(0)\mathscr{B}(\bar{u}^n-u(\bar{t}_n),\bar{\theta}_n)\\
		&\quad+2\Delta t_n\sum_{j=1}^{n-1}\mathscr{B}(\Delta t_jK(\bar{t}_n-\bar{t}_j)u(\bar{t}_j)-\int_{t_{j-1}}^{t_j}K(\bar{t}_n-s)u(s)\ {\rm ds},\bar\theta_n)\\
		&\quad+2\Delta t_n\mathscr{B}(\frac{\Delta t_n}{2}K(0)u(\bar{t}_n)-\int_{t_{n-1}}^{\bar t_n}K(\bar{t}_n-s)u(s)\ {\rm d}s,\bar{\theta}_n)\\
		&\quad+2\Delta t_n\mathscr{B}(F(\bar{t}_n)-\frac{1}{\Delta t_n}\int_{t_{n-1}}^{t_n}F(t)\ {\rm d}t,\bar{\theta}_n)\\
		&=\sum_{i=1}^{5}S_i.
	\end{align*}
	We define $ G(t)=K(\bar{t}_n-t)u(t) $ and proceed with standard techniques to estimate $ S_1-S_5 $, resulting in:
	\begin{align*}
		S_1&=\Delta t_n\sum_{j=1}^{n-1}K(\bar{t}_n-\bar{t}_j)\Delta t_j\mathscr{B}(\int_{t_{j-1}}^{\bar{t}_j}(t-t_{j-1})u_{tt}\ {\rm d}t,\bar{\theta}_n)\\
		&\quad+\Delta t_n\sum_{j=1}^{n-1}K(\bar{t}_n-\bar{t}_j)\Delta t_j\mathscr{B}(\int_{\bar t_{j}}^{t_j}(t-t_{j-1})u_{tt}\ {\rm d}t,\bar{\theta}_n)\\
		&\le C\Delta t\sum_{j=1}^{n-1}K(\bar{t}_n-\bar{t}_j)\Delta t^{\frac{5}{2}}(\int_{t_{j-1}}^{t_j}\|u_{tt}\|_a^2\ {\rm d}t)^{\frac{1}{2}}\|\bar{\theta}_n\|_a\\
		&\le C\Delta t^5\sum_{j=1}^{n-1}K(\bar{t}_n-\bar{t}_j)\int_{t_{j-1}}^{t_j}\|u_{tt}\|_a^2\ {\rm d}t+\frac{\mu\Delta t}{25}\|\bar{\theta}_n\|_a^2,\\
		S_2&=\Delta t_n^2K(0)\mathscr{B}(\frac{1}{2}\int_{t_{n-1}}^{\bar{t}_n}(t-t_{n-1})u_{tt}\ {\rm d}t+\frac{1}{2}\int_{\bar{t}_{n}}^{t_n}(t_n-t)u_{tt}\ {\rm d}t,\bar{\theta}_n)\\
		&\le C\Delta tK(0)\Delta t^{\frac{5}{2}}(\int_{t_{n-1}}^{t_n}\|u_{tt}\|_a^2\ {\rm d}t)^{\frac{1}{2}}\|\bar{\theta}_n\|_a \\
		& \le C\Delta t^4\int_{t_{n-1}}^{t_n}\|u_{tt}\|_a^2\ {\rm d}t+\frac{\mu\Delta t}{25}\|\bar{\theta}_n\|_a^2,\\
		S_3&=\Delta t_n\sum_{j=1}^{n-1}\mathscr{B}(\int_{t_{j-1}}^{t_j}(t-t_{j-1})(t_j-t)G_{tt}\ {\rm d}t,\bar{\theta}_n)\\
		&\quad-\Delta t_n\sum_{j=1}^{n-1}\Delta t_j\mathscr{B}(\int_{t_{j-1}}^{\bar{t}_j}(t-t_{j-1})G_{tt}\ {\rm d}t+\int_{\bar{t}_j}^{t_j}(t_j-t)G_{tt}\ {\rm d}t,\bar{\theta}_n)\\
		&\le C\Delta t\sum_{j=1}^{n-1}\Delta t^{5/2}(\int_{t_{j-1}}^{t_j}\|G_{tt}\|_a^2\ {\rm d}t)^{1/2}\|\bar{\theta}_n\|_a^2\\
		&\le C\Delta t^5\int_{0}^{t_{n-1}}\|G_{tt}\|_a^2\ {\rm d}t+\frac{\mu\Delta t}{25}\|\bar{\theta}_n\|_a^2,\\
		S_4&=2\Delta t_n\mathscr{B}(\int_{t_{n-1}}^{\bar{t}_n}(t-t_{n-1})G_t\ {\rm d}t,\bar{\theta}_n)\\
		&\le C\Delta t^{\frac{5}{2}}(\int_{t_{n-1}}^{t_n}\|G_t\|_a^2\ {\rm d}t)^{\frac{1}{2}}\|\bar{\theta}_n\|_a\le C\Delta t^4\int_{t_{n-1}}^{t_n}\|G_t\|_a^2\ {\rm d}t+\frac{\mu\Delta t}{25}\|\bar{\theta}_n\|_a^2,\\
		S_5&=\mathscr{B}(\int_{t_{n-1}}^{t_n}(t-t_{n-1})(t_n-t)F_{tt}\ {\rm d}t\\
		&\quad-\Delta t_n(\int_{t_{n-1}}^{\bar{t}_n}(t-t_{n-1})F_{tt}\ {\rm d}t+\int_{\bar{t}_n}^{t_n}(t_n-t)F_{tt}\ {\rm d}t),\bar{\theta}_n)\\
		&\le C\Delta t^4\int_{t_{n-1}}^{t_n}\|F_{tt}\|_a^2\ {\rm d}t+\frac{\mu\Delta t}{25}\|\bar{\theta}_n\|_a^2.
	\end{align*}
	
	Combining with all above estimate for $ S_i(1\le i\le 6) $, we conclude that 
	\begin{align*}
		R_5+R_6&\le C\Delta t^4\int_{t_{n-1}}^{t_n}(\|u_{tt}\|_a^2+\|G_t\|_a^2+\|F_{tt}\|_a^2)\ {\rm d}t
		+\frac{\mu\Delta t}{5}\|\bar{\theta}_n\|_a^2\\
		&\quad +C\Delta t^5\sum_{j=1}^{n-1}K(\bar{t}_n-\bar{t}_j)\int_{t_j}^{t_{j-1}}\|u_{tt}\|_a^2\ {\rm d}t
		+C\Delta t^5\int_{0}^{t_{n-1}}\|G_{tt}\|_a^2\ {\rm d}t.
	\end{align*}
	Finally, we move to bound the terms $ R_8-R_9 $. Letting 
	\begin{align*}
		H(t)=\int_0^tK(t-s)\rho(s)\ {\rm d}s,\qquad I(t)=K(\bar{t}_n-t)\rho(t),
	\end{align*}
	and using the definition of Ritz-Volterra projection, we can get
	\begin{align*}
		-2\Delta t_n \mathscr{A}(\bar{\rho}_n,\bar{\theta}_n)=\Delta t_n\mathscr{B}(
		H(t_n)+H(t_{n-1}),\bar{\theta}_n).
	\end{align*}
	Hence we can rewrite $ R_8+R_9 $ as: 
	\begin{align*}
		R_8+R_9&=2\Delta t_n\mathscr{B}(\bar{H}(t_n),\bar{\theta}_n)-2\Delta t_n\mathscr{B}(\bar{\rho}_{K}^{\Delta t,n},\bar{\theta}_n)\\
		&=2\Delta t_n\mathscr{B}(\bar{H}(t_n)-H(\bar{t}_n),\bar{\theta}_n)\\
		&\quad+2\Delta t_n\sum_{j=1}^{n-1}\mathscr{B}(\int_{t_{j-1}}^{t_j}K(\bar{t}_n-s)\rho(s)\ {\rm d}s-\Delta t_jK(\bar{t}_n-\bar{t}_j)\rho(\bar t_j),\bar{\theta}_n)\\
		&\quad+2\Delta t_n\mathscr{B}(\int_{t_{n-1}}^{\bar{t}_n}K(\bar{t}_n-s)\rho(s)\ {\rm d}s-\frac{\Delta t_n}{2}K(0)\rho(\bar{t}_n),\bar\theta_n)\\
		&\quad+2\Delta t_n\sum_{j=1}^{n-1}\mathscr{B}(\Delta t_jK(\bar{t}_n-\bar{t}_j)(\rho(\bar{t}_j)-\bar{\rho}_j),\bar{\theta}_n)\\
		&\quad+\Delta t_n^2K(0)\mathscr{B}(\rho(\bar{t}_n)-\bar\rho_n,\bar{\theta}_n).
	\end{align*}
	We employ a similar technique used to bound $ R_5+R_6 $ to estimate $ R_8+R_9 $, leading to:
	\begin{align*}
		R_8+R_9&\le C\Delta t^4\int_{t_{n-1}}^{t_n}(\|H_{tt}\|_a^2+\|I_t\|_a^2+\|\rho_{tt}\|_a^2)\ {\rm d}t+\frac{\mu\Delta t}{5}\|\bar\theta_n\|_a^2\\
		&\quad +C\Delta t^5\int_{0}^{t_{n-1}}\|I_{tt}\|_a^2\ {\rm d}t+C\Delta t^5\sum_{j=1}^{n-1}K(\bar{t}_n-\bar{t}_j)\int_{t_{n-1}}^{t_n}\|\rho_{tt}\|_a^2\ {\rm d}t.
	\end{align*}
	Now, we  sum $ \eqref{fem_proof_eq} $ from $ n=1 $ to $ N $ and use all above estimate for $ R_i(1\le i\le 9) $ to conclude
	\begin{align}\label{fem_proof_2}
		\begin{split}
			&\|\theta_N\|^2+2\sum_{n=1}^{N}\Delta t_n\|\bar{\theta}_n\|_a^2+2\sum_{n=1}^{N}\Delta t_n\mathscr{B}(\bar{\theta}_{K}^{\Delta t,n},\bar{\theta}_n)-\mu\Delta t\sum_{n=1}^{N}\|\bar{\theta}_n\|_a^2\\
			&\quad\le C\Delta t^4\int_{0}^{T}(\|f_{tt}\|_a^2+\|u_{tt}\|_a^2+\|F_{tt}\|_a^2+\|G_{t}\|_a^2+\|G_{tt}\|_a^2)\ {\rm d}t+\|\theta_0\|^2
			\\
			&\qquad +C\Delta t^4\int_{0}^{T}(\|H_{tt}\|_a^2+\|I_{t}\|_a^2+\|I_{tt}\|_a^2+\|\rho_{tt}\|_a^2)\ {\rm d}t
			+C\Delta t_n\sum_{n=1}^{N}\|\partial_t^+\rho_n\|^2	.
		\end{split}
	\end{align}
	Since 
	\begin{align*}
		&2\sum_{n=1}^{N}\Delta t_n\|\bar{\theta}_n\|_a^2+2\sum_{n=1}^{N}\Delta t_n\mathscr{B}(\bar{\theta}_K^{\Delta t,n},\bar{\theta}_n)-\mu\Delta t\sum_{n=1}^{N}\|\bar{\theta}_n\|_a^2\\
		&\quad=2\sum_{n=1}^{N}\Delta t_n\|\bar{\theta}_n\|_a^2
		-\mu\Delta t\sum_{n=1}^{N}\|\bar{\theta}_n\|_a^2
		\\
		&\qquad+2\sum_{n=1}^{N}\Delta t_n\mathscr{B}(-\frac{\Delta t_n}{2}K(0)\bar{\theta}_n+\sum_{j=1}^{n-1}\Delta t_jK(\bar{t}_n-\bar{t}_j)\bar{\theta}_j,\bar{\theta}_n)\\
		&\quad\ge 2\sum_{n=1}^{N}\Delta t_n\|\bar{\theta}_n\|_a^2
		-\mu\Delta t\sum_{n=1}^{N}\|\bar{\theta}_n\|_a^2
		\\
		&\qquad-2\sum_{n=1}^{N}\Delta t_nc_0\sum_{j=1}^{n}\Delta t_jK(\bar{t}_n-\bar{t}_j)\|\bar{\theta}_n\|_a\|\bar{\theta}_j\|_a\\
		&\quad\ge 2\sum_{n=1}^{N}\Delta t_n\|\bar{\theta}_n\|_a^2
		-\mu\Delta t\sum_{n=1}^{N}\|\bar{\theta}_n\|_a^2
		\\
		&\qquad-\sum_{n=1}^{N}\Delta t_nc_0\sum_{j=1}^{n}\Delta t_jK(\bar{t}_n-\bar{t}_j)(\|\bar{\theta}_n\|_a^2+\|\bar{\theta}_j\|_a^2)\\
		&\quad\ge (2-2c_0K_0)\sum_{n=1}^{N}\Delta t_n\|\bar{\theta}_n\|_a^2-\mu\Delta t\sum_{n=1}^{N}\|\bar{\theta}_n\|_a^2,
	\end{align*}
	we choose $ \mu\in (0,2C_1^{-1}(1-c_0K_0))  $ such that 
	\begin{align*}
		\mu\Delta t\sum_{n=1}^{N}\|\bar{\theta}_n\|_a^2\le C_1\mu(\delta t)\sum_{n=1}^{N}\|\bar{\theta}_n\|_a^2\le 2(1-c_0K_0)\sum_{n=1}^{N}\Delta t_n\|\bar{\theta}_n\|_a^2,
	\end{align*}
	and notice that
	\begin{align*}
		C\sum_{n=1}^{N}\Delta t_n\|\partial_t^+\rho_n\|^2=C\frac{1}{\Delta t_n}\sum_{n=1}^{n}\int_{\Omega}(\int_{t_{n-1}}^{t_n}\rho_t\ {\rm d}t)^2\ {\rm d} x
		\le C\int_0^T\|\rho_t\|^2\ {\rm d}t,
	\end{align*}
	therefore $ \eqref{fem_proof_2} $ becomes 
	\begin{align}\label{fem_proof_3}
		\|\theta_N\|^2\le C\Delta t^4+\int_0^t\|\rho_t\|^2\ {\rm d}t.
	\end{align}
	By applying \Cref{lemma505} and \eqref{fem_proof_3}, and employing the triangle inequality, we arrive at our final conclusion.
\end{proof}
Now, we will proceed to establish the error estimation between the solution of the standard finite element method \eqref{discrete-eq0} and our novel scheme \eqref{ISVD_eq1}.

\begin{lemma}\label{error_uh_uhat_another}
	Let $u_h^{n}$ and $\widehat{u}_h^{n}$ be the solution of \eqref{discrete-eq0} and \eqref{ISVD_eq1}, respectively. Given Assumption \ref{A2}, the following error bound is established:
	\begin{align*}
		\|u_h^N-\widehat{u}_h^N\|\le 
		\sqrt{T(1+\gamma^{-1})}\max_{1\le n\le N}\max_{1\le j \le n-1}\| \bar{\widetilde{u}}_h^{n-1,j}-\bar{\widehat{u}}_h^j\|_a,
	\end{align*}
	where 
	$  \gamma\in (0,2c_0^{-1}K_0^{-1}-2)   $ and $ c_0,K_0 $  are defined in  Assumption \ref{A2}.
\end{lemma} 
\begin{proof}
	Recall that $u_h^{n}$ and $\widehat  u_h^{n}$ satisfy the following equations
	\begin{subequations}
		\begin{align}
			(\partial_t^+u_h^n,v_h)+\mathscr A(\bar{u}_h^n,v_h)+\mathscr B(\bar{u}_{h,K}^{\Delta t,n},v_h)&=(\bar{f}^n,v_h)	,\label{fem_so1}\\	
			(\partial_t^+\widehat{u}_h^n,v_h)+\mathscr A(\bar{\widehat{u}}_h^n,v_h)+\mathscr B(M^n,v_h)&=(\bar{f}^n,v_h)		.\label{ISVD_so1}
		\end{align}
	\end{subequations}
	Subtracting    \eqref{fem_so1}  from $ \eqref{ISVD_so1} $ and introducing the following notations:  
	\begin{align*}
		\widehat{e}_{n}&=u_{h}^{n}-\widehat{u}_{h}^{n},\qquad\quad \ \widetilde{e}_{n,j}=u_{h}^{j}-\widetilde{u}_{h}^{n,j},\\
		\bar{\widetilde{e}}_{n,j}&=\frac{\widetilde{e}_{n,j}+\widetilde{e}_{n,j-1}}{2},\qquad
		\bar{\widehat{e}}_{n}=\frac{\widehat{e}_n+\widehat{e}_{n+1}}{2},
	\end{align*} 
	we  obtain the following error equation:
	\begin{align}\label{eq6161}
		\begin{split}
			&\left(\partial_t^+\widehat{e}_n,v_h\right)+ \mathscr{A}(\bar{\widehat{e}}_{n},v_h)+\sum_{j=1}^{n-1}\Delta t_jK(\bar{t}_n-\bar{t}_j)\mathscr{B}(\bar{\widetilde{e}}_{n-1,j},v_h)\\
			&\quad+\frac{\Delta t_n}{2}K(0)\mathscr{B}(\bar{\widehat{e}}_n,v_h)=0.
		\end{split}
	\end{align}
	Substituting $v_h = 2\Delta t_n\bar{\widehat{e}}_{n} \in V_h$ into the above equation, we derive the following equation
	\begin{align*}
		&\|\widehat{e}_n\|^2-\|\widehat{e}_{n-1}\|^2+2\Delta t_n\|\bar{\widehat{e}}_n\|_a^2\\
		&\quad=-2\Delta t_n\sum_{j=1}^{n-1}\Delta t_jK(\bar{t}_n-\bar{t}_j)\mathscr{B}(\bar{\widetilde{e}}_{n-1,j},\bar{\widehat{e}}_n) -\Delta t_n^2K(0)\mathscr{B}(\bar{\widehat{e}}_n,\bar{\widehat{e}}_n).
	\end{align*}
	Using  Cauchy-Schwarz inequality, Assumption \ref{A2}, we can deduce the following equation
	\begin{align*}
		&\|\widehat{e}_n\|^2-\|\widehat{e}_{n-1}\|^2+2\Delta t_n\|\bar{\widehat{e}}_n\|_a^2\\
		&\quad\le 2c_0\Delta t_n\sum_{j=1}^{n-1}\Delta t_jK(\bar{t}_n-\bar{t}_j)\|\bar{\widetilde{e}}_{n-1,j}\|_a\|\bar{\widehat{e}}_n\|_a +c_0\Delta t_n^2K(0)\|\bar{\widehat{e}}_n\|_a^2
		\\
		&\quad\le c_0\Delta t_n\sum_{j=1}^{n-1}\Delta t_jK(\bar{t}_n-\bar{t}_j)(\|\bar{\widetilde{e}}_{n-1,j}\|_a^2+\|\bar{\widehat{e}}_n\|_a^2)+c_0\Delta t_n^2K(0)\|\bar{\widehat{e}}_n\|_a^2\\
		&\quad\le c_0\Delta t_n\sum_{j=1}^{n-1}\Delta t_jK(\bar{t}_n-\bar{t}_j)((\|\bar{\widehat{e}}_j-\bar{\widetilde{e}}_{n-1,j}\|_a+\|\bar{\widehat{e}}_j\|_a)^2+\|\bar{\widehat{e}}_n\|_a^2)  +c_0\Delta t_n^2K(0)\|\bar{\widehat{e}}_n\|_a^2\\
		&\quad\le c_0\Delta t_n\sum_{j=1}^{n}\Delta t_jK(\bar{t}_n-\bar{t}_j)
		((1+\gamma)\|\bar{\widehat{e}}_j\|_a^2+(1+\gamma^{-1})\|\bar{\widetilde{e}}_{n-1,j}-\bar{\widehat{e}}_j\|_a^2+\|\bar{\widehat{e}}_n\|_a^2),
	\end{align*}
	for some  $ \gamma \in (0,1) $. Summing over $n$ ranging from $1$ to $N$, we have:
	\begin{align}
		\label{main_proof5061}
		\begin{split}
			&\|\widehat{e}_n\|^2-\|\widehat{e}_0\|^2+2\Delta t_n\sum_{n=1}^{N}\|\bar{\widehat{e}}_n\|_a^2\\
			&\quad\le c_0\sum_{n=1}^{N}\Delta t_n\sum_{j=1}^{n}\Delta t_jK(\bar{t}_n-\bar t_j)(1+\gamma)\|\bar{\widehat{e}}_j\|_a^2\\
			&\qquad +c_0\sum_{n=1}^{N}\Delta t_n\sum_{j=1}^{n}\Delta t_jK(\bar{t}_n-\bar t_j)\|\bar{\widehat{e}}_n\|_a^2\\
			&\qquad + c_0\sum_{n=1}^{N}\Delta t_n\sum_{j=1}^{n}\Delta t_jK(\bar{t}_n-\bar t_j)(1+\gamma^{-1})\|\bar{\widetilde{e}}_{n-1,j}-\bar{\widehat{e}}_j\|_a^2\\
			& \le (2+\gamma)c_0K_0\sum_{n=1}^{N}\Delta t_n\|\bar{\widehat{e}}_n\|_a^2\\
			&\qquad +(1+\gamma^{-1})c_0K_0\sum_{n=1}^{N}\Delta t_n{\small \max_{1\le j\le n-1}}\|\bar{\widetilde{e}}_{n-1,j}-\bar{\widehat{e}}_j\|_a^2.
		\end{split}
	\end{align}
	Using Assumption \ref{A2}, we can choose  $  \gamma\in (0,2c_0^{-1}K_0^{-1}-2)   $ such that
	\begin{align*}
		(2+\gamma)c_0K_0<2.
	\end{align*}
	Then  \eqref{main_proof5061}  becomes 
	\begin{align*}
		\|\widehat{e}_N\|^2&\le \|\widehat{e}_0\|^2+(1+\gamma^{-1})c_0K_0T\max_{1\le n\le N}\max_{1\le j \le n-1}\|\bar{\widetilde{e}}_{n-1,j}-\bar{\widehat{e}}_j\|_a^2\\
		&\le (1+\gamma^{-1})T\max_{1\le n\le N}\max_{1\le j \le n-1}\|\bar{\widetilde{e}}_{n-1,j}-\bar{\widehat{e}}_j\|_a^2.
	\end{align*}
	
	The proof is finalized by utilizing $\widehat{e}_0 = 0$ along with the condition $c_0K_0 < 1$ (Assumption \ref{A2}).
\end{proof} 
Next we turn to estimate the  term $\displaystyle\max_{1\le n\le N}\max_{1\le j \le n-1}\|\widetilde{u}_h^{n-1,j}-\widehat{u}_h^j\|_a  $. We have the following error bound:

\begin{lemma}\label{maxmax}
	Let $\{\widehat{u}_h^n\}_{n=1}^N$ be the solution of \eqref{ISVD_eq1}, and let $\{\widetilde{u}_h^{n,j}\}_{j=0}^{n}$ represent the compressed data corresponding to $\{ \widetilde{u}_h^{n-1,0}, \widetilde{u}_h^{n-1,1},\ldots, \widetilde{u}_h^{n-1,n-1}, \widehat{u}_h^{n}\}$. This compressed solution is obtained using the incremental SVD method with a tolerance of \textup{\texttt{tol}} applied to both $p$-truncation and singular value truncation. Let $ T_{sv} $ represent the total number of times the singular
	value truncation is applied, then we can obtain the following inequality:
	\begin{align*}
		\max_{1\le n\le N}\max_{1\le j \le n-1}\|\widetilde{u}_h^{n,j}-\widehat{u}_h^j\|_a\le (T_{sv} +1)\sqrt{\sigma(A)}\textup{\texttt{tol}},
	\end{align*}
	where $\sigma(A)$ represents the spectral radius of the stiffness matrix $A$.
\end{lemma} 

\begin{proof}
	Assuming $\widehat u_j$ and $\widetilde u_{k,\ell}$ are the coefficients of $\widehat {u}_h^{j}$ and $\widetilde u_h^{k,\ell}$ corresponding to the finite element basis functions $\{\phi_s\}_{s=1}^m$, respectively, we establish the following inequality for $0\le j\le n-1\le N-1$:
	\begin{align*}
		\|\widetilde{u}_h^{n,j}-\widehat{u}_h^j\|_a &\le 	\sum_{k=0}^{n-j-1}\|\widetilde{u}_h^{n-k,j}-\widetilde{u}_h^{n-k-1,j}\|_a + \|\widetilde{u}_h^{j,j}-\widehat u_h^j\|_a\\
		& = \sum_{k=0}^{n-j-1}   \sqrt{ (\widetilde u_{n-k,j} - \widetilde u_{n-k-1,j})^\top A (\widetilde u_{n-k,j} - \widetilde u_{n-k-1,j})} \\
		&\quad +     \sqrt{ (\widetilde u_{j,j} - \widehat u_{j})^\top A (\widetilde u_{j,j} - \widehat u_{j})} \\
		&  \le \sum_{k=0}^{n-j-1} \sqrt{\sigma(A)} |\widetilde u_{n-k,j} - \widetilde u_{n-k-1,j}| +  \sqrt{\sigma(A)} |\widetilde u_{j,j} - \widehat u_{j}|.
	\end{align*}
	
	Here, $\sigma(A)$ represents the spectral radius of the stiffness matrix $A$. It is  notable that $\widetilde u_{n-k,j}$ corresponds to the $j$-th compressed data at the $n-k$-th step, as illustrated by:
	\begin{align*}
		[\widetilde u_{n-k-1,0}\mid \widetilde u_{n-k-1,1} \mid \ldots \mid\widetilde u_{n-k-1,n-k-2}\mid \widehat u_{n-k-1}]\\
		\stackrel{\text{Compressed}}{\longrightarrow } [\widetilde u_{n-k,0}\mid \widetilde u_{n-k,1}\mid\ldots\mid \widetilde u_{n-k,n-k}].
	\end{align*}
	
	Furthermore, considering that both $p$ truncation and no truncation maintain the prior data unchanged, it follows that at most $\min\{T_{sv},n-j-1\} $ terms of $\{\widetilde u_{n-k,j}- \widetilde u_{n-k-1,j}\}_{k=0}^{n-j-1}$ are non-zero, where $ T_{sv} $ represents the total number of times singular value truncation is applied. Consequently, for any $0\le j\le n-1\le N-1$, employing \Cref{error_SingularValueTruncation}, we can derive:
	\begin{align*}
		\max_{1\le n\le N}\max_{1\le j \le n-1}\|\widetilde{u}_h^{n,j}-\widehat{u}_h^j\|_a\le (T_{sv} +1)\sqrt{\sigma(A)} \texttt{tol}.
	\end{align*}
\end{proof}

Hence, by applying the triangle inequality to \Cref{fem_error,error_uh_uhat_another,maxmax}, we can derive the error estimate for $\|u(t_n)-\widehat{u}_{h}^{n} \| $. This completes the proof of \Cref{main_res}.
\section{Singular Kernel Case}\label{singularkernel}
In this section, we extend the application of our new method to a more general scenario wherein the memory kernel exhibits singularity. It is crucial to note that the primary divergence between these two cases—singular and non-singular kernels—lies in the treatment of the convolution quadrature for the memory term. We reiterate that the equations under consideration in this paper are as follows:
\begin{align}\label{singular-integro-differential equation}
	u_t+\mathcal Au+\int_0^t K(t-s) \mathcal Bu(s)\ {\rm d}s&=f(x,t)
\end{align}
with singular kernel 
\begin{align*}
	K(t)=e^{-\lambda t}\frac{1}{\Gamma(\alpha)}t^{\alpha-1},
\end{align*}
where $ \alpha\in (0,1) $. It is well known that when $ \lambda=0 $, the kernel is Abel kernel $ K(t)=t^{\alpha-1}/\Gamma(\alpha) $.  To obtain the numerical scheme of \eqref{singular-integro-differential equation}, we use the second order  convolution quadrature rule from \cite{QiuWenlin2023AFEG}:
\begin{align*}
	Q_{t_n}^{(\alpha,\lambda)}(\phi)=\Delta t^{\alpha}\sum_{p=0}^{n}\mathcal{X}_p^{(\alpha,\lambda)}\phi(t_n-t_p)+\varpi_n^{(\alpha,\lambda)}\phi(0),
\end{align*}
where 
\begin{align*}
	&\mathcal{X}_n^{(\alpha,\lambda)}=e^{-\lambda t_n}\mathcal{X}_n^{(\alpha,0)},\qquad
	\varpi_n^{(\alpha,\lambda)}=\frac{e^{-\lambda t_n}t_n^{\alpha}}{\mathcal{X}_n^{(\alpha,0)}}-\Delta t^{\alpha}\sum_{p=0}^{n}e^{-\lambda(t_n-t_p)}\mathcal{X}_p^{(\alpha,\lambda)},\\
	&\mathcal{X}_n^{(\alpha,0)}=(\frac{3}{2})^{-\alpha}\sum_{s=0}^{n}3^{-s}\sigma_s^{(\alpha)}\sigma_{n-s}^{(\alpha)},\qquad
	\sigma_s^{(\alpha)}=\frac{\Gamma(\alpha+s)}{\Gamma(\alpha)\Gamma(s+1)}.
\end{align*}

\begin{remark}
	During the coding process, double-precision floating-point numbers are commonly employed to store and process data. It's crucial to acknowledge that the representation range for double-precision floating-point numbers is approximately $ \pm 1.8\times 10^{308} $, while $ \Gamma(172)\approx 1.25\times10^{309} $. Consequently, handling the term $\sigma_s^{(\alpha)} $ necessitates careful consideration. In this paper, we leverage the property of the Gamma function $ \Gamma(1+k)=k\Gamma(k) $ to obtain:
	\begin{align*}
		\sigma_s^{(\alpha)}&=\frac{(\alpha+s-1)(\alpha+s-2)\cdots(\alpha)\Gamma(\alpha)}{\Gamma(\alpha)s(s-1)\cdots 1\cdot\Gamma(1)}=(\frac{\alpha+s-1}{s})(\frac{\alpha+s-2}{s-1})\cdots \frac{\alpha}{1}.
	\end{align*}
\end{remark}

Then, we derive the numerical scheme for \eqref{singular-integro-differential equation}, comprising the Continuous Galerkin method for space discretization and the BDF2 scheme for temporal discretization. Specifically, given $ u_h^0\in V_h $, the objective is to find $ u_h^1\in V_h $ and $ u_h^n(2\le n\le N)\in V_h $ such that:
\begin{subequations}\label{singular-numerical scheme}
	\begin{align}
		(\partial_t^+u_h^1,v_h)+\mathscr{A}(u_h^1,v_h)+\Delta t^{\alpha}\sum_{p=0}^{1}\mathcal{X}_p^{(\alpha,\lambda)}\mathscr{B}(u_h^{1-p},v_h)+\varpi_1^{(\alpha,\lambda)}\mathscr{B}(u_h^0,v_h)&=(f^1,v_h),\\
		(D_t^{(2)}u_h^n,v_h)+\mathscr{A}(u_h^n,v_h)+\Delta t^{\alpha}\sum_{p=0}^{n}\mathcal{X}_p^{(\alpha,\lambda)}\mathscr{B}(u_h^{n-p},v_h)+\varpi_n^{(\alpha,\lambda)}\mathscr{B}(u_h^0,v_h)&=(f^n,v_h),
	\end{align}
\end{subequations}
with
\begin{align*}
	D_t^{(2)}u_h^n=\frac{3}{2}\partial_t^+u_h^n-\frac{1}{2}\partial_t^+u_h^{n-1},\qquad
	\partial_t^+u_h^n=\frac{u_h^n-u_h^{n-1}}{\Delta t}.
\end{align*}

Now, we introduce our new method by incorporating the incremental SVD algorithm, resulting in the following objective: given $ u_h^0\in V_h $, the goal is to find $ \widehat u_h^1\in V_h $ and $ \widehat u_h^n(2\le n\le N)\in V_h $ such that:
\begin{subequations}\label{singular-fem-svd-scheme}
	\begin{align}
		(\partial_t^+\widehat{u}_h^1,v_h)+\mathscr{A}(\widehat{u}_h^1,v_h)+\Delta t^{\alpha}\sum_{p=0}^{1}\mathcal{X}_p^{(\alpha,\lambda)}\mathscr{B}(\widehat{u}_h^{1-p},v_h)+\varpi_1^{(\alpha,\lambda)}\mathscr{B}(\widehat{u}_h^0,v_h)&=(f^1,v_h),\\
		\begin{split}
			(D_t^{(2)}\widehat{u}_h^n,v_h)+\mathscr{A}(\widehat{u}_h^n,v_h)+\Delta t^{\alpha}\sum_{p=1}^{n}\mathcal{X}_p^{(\alpha,\lambda)}\mathscr{B}(\widetilde{u}_h^{n-1,n-p},v_h)&\\
			+\Delta t^{\alpha}\mathcal{X}_0^{(\alpha,\lambda)}\mathscr{B}(\widehat u_h^n,v_h)+\varpi_n^{(\alpha,\lambda)}\mathscr{B}(\widehat{u}_h^{0},v_h)&=(f^n,v_h).
		\end{split}
	\end{align}
\end{subequations}
Similar to \Cref{Error_estimate}, we initially present the main result, followed by a two-step proof strategy. The approach involves deriving the error bound between \eqref{singular-integro-differential equation} and \eqref{singular-numerical scheme}, and subsequently estimating the error between \eqref{singular-numerical scheme} and \eqref{singular-fem-svd-scheme}.

Now, we demonstrate our main result :
\begin{theorem}\label{main_result 2}
	Let $u$ and $\widehat {u}_h^N$ denote the solution of \eqref{integro-diff-Eq} and \eqref{singular-fem-svd-scheme}, respectively. Throughout the entire process of the incremental SVD algorithm, the tolerance \textup{\texttt{tol}} is applied to both $p$-truncation and singular value truncation.  If $u(t)\in H^{k+1}(\Omega)$ and assume that $ \mathscr{B} $ is nonnegative symmetric, then the following error bound holds:
	\begin{align}\label{error_est1}
		\| u(T)-\widehat {u}_h^N\|\le C(h^{k+1}+\Delta t^{1+\alpha})+C( {T_{sv}}+1)\sqrt{T\sigma(A)}\mathtt{tol},
	\end{align}
	where $C$ denotes a  positive constants, independent of $h$, $\Delta t$, and \textup{\texttt{tol}}, and $\sigma(A)$ represents the spectral radius of the stiffness matrix $A$, $ T_{sv}$ signifies  the total number of times singular value truncation is applied.
\end{theorem}
The following three lemmas are very useful in our proof process.
\begin{lemma}\label{bdf2-sum}
	\textup{\cite[Lemma 4.2]{QiuWenlin2020Atta}}
	For any $w^n \in V_h, 1 \leq n \leq N$, it follows that
	\begin{align*}
		\Delta t(\partial_t^+ w^1, w^1)+\Delta t \sum_{n=2}^N (D_t^{(2)} w^n, w^n )\geq \frac{3}{4}\|w^N\|^2-\frac{1}{4}\left(\|w^{N-1}\|^2+\|w^1\|^2+\|w^0\|^2\right).
	\end{align*}
\end{lemma}
\begin{lemma}\label{kernel-positive}
	\textup{\cite[Lemma 8]{QiuWenlin2023AFEG}} 
	Let $\mathcal{X}_p^{(\alpha, \lambda)}$ be the coefficients given above. Then, for any $N \in \mathbb{Z}_{+}$ and $
	\left[\widehat{v}^0, \widehat{v}^1, \ldots, \widehat{v}^N\right] \in \mathbb{R}^{N+1},$
	we have
	\begin{align}\label{positive-definite}
		\sum_{n=0}^N\left(\sum_{p=0}^n \mathcal
		X_p^{(\alpha, \lambda)} \widehat{v}^{n-p}\right) \widehat{v}^n \geq 0.
	\end{align}
\end{lemma}
\begin{lemma}\label{quadrature-error}
	\textup{\cite[Lemma 7]{QiuWenlin2023AFEG}}
	If $\varphi \in C^1([0, T])$ and $\varphi^{\prime \prime}$ is continuous and integrable on $(0, T]$, then, for $\alpha \in(0,1)$ and $\lambda \in[0, \infty)$, we have
	\begin{align*}
		\left|\hat{\varepsilon}_k^{(\alpha, \lambda)}(\varphi)\left(t_n\right)\right| \leq & C e^{-\lambda t_n}\left[\Delta t^2\left|\varphi^{\prime}(0)\right| t_n^{\alpha-1}+\Delta t^{\alpha+1} \int_{t_{n-1}}^{t_n}\left|(p \varphi)^{\prime \prime}(s)\right| d s\right. \\
		& \left.+\Delta t^2 \int_0^{t_{n-1}}\left(t_n-s\right)^{\alpha-1}\left|(p \varphi)^{\prime \prime}(s)\right| d s\right], \quad n=1,2, \ldots, N,
	\end{align*}
	where 
	\begin{align*}
		p(t)=e^{\lambda t},\ \varepsilon_k^{(\alpha,\lambda)}(\phi)(t_n)=Q_{t_n}^{(\alpha,\lambda)}(\phi)-\int_{0}^{t_n}K(t_n-s)\phi(s) {\rm d}s.
	\end{align*}
\end{lemma}
\begin{remark}
	From \Cref{quadrature-error}, it follows that if $\left|\varphi^{\prime}(0)\right|=0$ and $\left|\varphi^{\prime \prime}(t)\right| \leq C$ for $t \in[0, T]$, then the quadrature error is $\mathcal{O}(\Delta t^2)$.
\end{remark}
\begin{lemma}\label{singular-lemme-error-fem}
	Suppose $u(t)\in H_0^{k+1}(\Omega) $ and $u_h^n$ are the solutions of \eqref{singular-integro-differential equation} and \eqref{singular-numerical scheme}, respectively. Assume that  $ \left\| u_0-u_h^0\right\|\le Ch^{k+1}\left\| u_0\right\|_{k+1}   $.  If $ \mathscr{B} $ is nonnegative symmetric and $ u $ satisfies the following regularity conditions :
	\begin{align*}
		\left\|u_t\right\| \leq C, \quad\left\|u_{t t}\right\| \leq C e^{-\lambda t} t^{\alpha-1}, \quad\left\|u_{t t t}\right\| \leq C e^{-\lambda t} t^{\alpha-2}, \quad t \rightarrow 0^{+},
	\end{align*}
	then 
	\begin{align*}
		\max _{1 \leq n \leq N}\left\|u^n-u_h^n\right\| \leq & C(h^{k+1}+\Delta t^{1+\alpha}),
	\end{align*}
	where $ C $ denotes a positive constant, independent of $ h $ and $ \Delta t $.
\end{lemma}
The proof of \Cref{singular-lemme-error-fem} is similar to  \cite{QiuWenlin2023AFEG}, hence we omit it there.
\begin{lemma}\label{singular-isvd-error}
	Let $ \widehat{u}_h^n $ and $u_h^n$ be the solutions of $ \eqref{singular-fem-svd-scheme} $ and \eqref{singular-numerical scheme}, respectively. Given Assumption \ref{A2}, suppose  \textup{\texttt{tol}} to be applied to both $p$-truncation and singular value truncation and  $ T_{sv} $ represent the total number of times the singular value truncation is applied, then there exists a positive $ C $, independent of $ h $ and $ \Delta t $, such that 
	\begin{align*}
		\|\widehat{u}_h^N-u_h^N\|\le C\sqrt{T\sigma(A)}(T_{sv}+1)\mathtt{tol}.
	\end{align*}
\end{lemma}
\begin{proof}
	We recall that $ \widehat{u}_h^n $ and $ u_h^n $ satisfying equations \eqref{singular-fem-svd-scheme} and $ \eqref{singular-numerical scheme} $, introducing the notations 
	\begin{align*}
		\widehat{e}_n=u_h^n-\widehat{u}_h^n,  \ \widetilde{e}_{n,j}=u_h^j-\widetilde{u}_h^{n,j},
	\end{align*}  
	and subtract \eqref{singular-fem-svd-scheme} from \eqref{singular-numerical scheme} to obtain 
	\begin{subequations}
		\begin{align}\label{eq-error-isvd-fem1}
			(\partial_t^+\widehat{e}_1,v_h)+\mathscr{A}(\widehat{e}_1,v_h)+\Delta t^{\alpha}\sum_{p=0}^{1}\mathcal{X}_p^{(\alpha,\lambda)}\mathscr{B}(\widehat{e}_{1-p},v_h)+\varpi_1\mathscr{B}(\widehat{e}_0,v_h)&=0,\\
			\begin{split}
				\label{eq-error-isvd-fem2}
				(D_t^{(2)}\widehat{e}_n,v_h)+\mathscr{A}(\widehat{e}_n,v_h)+\Delta t^{\alpha}\sum_{p=1}^{n}\mathcal{X}_{p}^{(\alpha,\lambda)}\mathscr{B}(\widetilde{e}_{n-1,n-p},v_h)&\\
				+\Delta t^{\alpha}\mathcal{X}_0^{(\alpha,\lambda)}\mathscr{B}(\widehat{e}_n,v_h)+\varpi_n^{(\alpha,\lambda)}\mathscr{B}(\widehat e_0,v_h)&=0.
			\end{split}
		\end{align}
	\end{subequations}
	Taking $ v_h=\Delta t\widehat{e}_1 $ in \eqref{eq-error-isvd-fem1} and $ v_h=\Delta t\widehat{e}_n $ in \eqref{eq-error-isvd-fem2}, summing \eqref{eq-error-isvd-fem2} from $ n=2 $ to $ n=N $ with  \eqref{eq-error-isvd-fem1}, it follows that 
	\begin{align*}
		&\Delta t(\partial_t^+\widehat{e}_1,\widehat{e}_1)
		+\Delta t\sum_{n=2}^{N}(D_t^{(2)}\widehat{e}_n,\widehat{e}_n)
		+\Delta t\sum_{n=1}^{N}\|\widehat{e}_n\|_a^2\\
		&\quad +\Delta t^{1+\alpha}\sum_{n=1}^{N}\sum_{p=0}^{n}\mathcal{X}_p^{(\alpha,\lambda)}\mathscr B(\widehat{e}_{n-p},\widehat{e}_n)\\
		&=-\Delta t\varpi_1^{(\alpha,\lambda)}\mathscr{B}(\widehat{e}_0,\widehat{e}_1)-\Delta t\sum_{n=2}^{N}\varpi_n^{(\alpha,\lambda)}\mathscr{B}(\widehat{e}_0,\widehat{e}_n)\\
		&\quad +\Delta t^{1+\alpha}\sum_{n=2}^{N}\sum_{p=1}^{n}\mathcal{X}_p^{(\alpha,\lambda)}\mathscr{B}(\widehat{e}_{n-p}-\widetilde{e}_{n-1,n-p},\widehat{e}_n).
	\end{align*}
	By  \Cref{bdf2-sum}  and using the fact $ \widehat{e}_0=0 $, it is obvious that 
	\begin{align*}
		&\frac{3}{4}\|\widehat{e}_N\|^2-\frac{1}{4}(\|\widehat{e}_{N-1}\|^2+\|\widehat{e}_1\|^2+\|\widehat{e}_0\|^2)+\Delta t\sum_{n=1}^{N}\|\widehat{e}_n\|_a^2\\
		&\quad +\Delta t^{1+\alpha}\sum_{n=1}^{N}\sum_{p=0}^{n}\mathcal{X}_p^{(\alpha,\lambda)}\mathscr{B}(\widehat{e}_{n-p},\widehat{e}_n)\\
		&\le -\Delta t\varpi_n^{(\alpha,\lambda)}\mathscr{B}(\widehat{e}_0,\widehat{e}_n)
		+\Delta t^{1+\alpha}\sum_{n=2}^{N}\sum_{p=1}^{n}\mathcal{X}_p^{(\alpha,\lambda)}\mathscr{B}(\widehat{e}_{n-p}-\widetilde{e}_{n-1,n-p},\widehat{e}_n)\\
		&\le c_0\Delta t^{1+\alpha}\sum_{n=2}^{N}\sum_{p=1}^{n}\mathcal{X}_p^{(\alpha,\lambda)}
		\|\widehat{e}_{n-p}-\widetilde{e}_{n-1,n-p}\|_a\|\widehat{e}_n\|_a.
	\end{align*}
	Selecting an integer $1\le  k\le N $ such that $\displaystyle \|\widehat{e}_k\|=\max_{1 \leq n \leq N}\|\widehat{e}_n\| $, then inequality above becomes
	\begin{align*}
		&\frac{1}{4}\|\widehat{e}_k\|^2+\Delta t\sum_{n=1}^{k}\|\widehat{e}_n\|_a^2+\Delta t^{1+\alpha}\sum_{n=1}^{k}\sum_{p=0}^{n}\mathcal{X}_p^{(\alpha,\lambda)}\mathscr{B}(\widehat{e}_{n-p},\widehat{e}_n)\\
		&\le\frac{3}{4}\|\widehat{e}_k\|^2-\frac{1}{4}(\|\widehat{e}_{k-1}\|^2+\|\widehat{e}_1\|^2)+\Delta t\sum_{n=1}^{k}\|\widehat{e}_n\|_a^2\\
		&\quad +\Delta t^{1+\alpha}\sum_{n=1}^{k}\sum_{p=0}^{n}\mathcal{X}_p^{(\alpha,\lambda)}\mathscr{B}(\widehat{e}_{n-p},\widehat{e}_n)\\
		&\quad\le c_0\Delta t^{1+\alpha}\sum_{n=2}^{N}\sum_{p=1}^{n}\mathcal{X}_p^{(\alpha,\lambda)}
		(\frac{1}{2\varepsilon}	\|\widehat{e}_{n-p}-\widetilde{e}_{n-1,n-p}\|_a^2
		+\frac{\varepsilon}{2}\|\widehat{e}_n\|_a^2),
	\end{align*}
	for some $ \varepsilon\in (0,1). $
	
	From \cite{MR1686149,ChenC.1992FEAo} we can follow  if  $ \mathscr{B} $ is nonnegative symmetric and $ \mathcal{X}_p^{(\alpha,\lambda)} $ satisfies positive definite condition \eqref{positive-definite} that 
	\begin{align*}
		\sum_{n=1}^{k}\sum_{p=0}^{n}\mathcal{X}_p^{(\alpha,\lambda)}\mathscr{B}(\widehat{e}_{n-p},\widehat{e}_n)\ge 0.
	\end{align*}
	From \cite[Lemma 6]{QiuWenlin2023AFEG}, we know $ \mathcal{X}_n^{(\alpha,\lambda)}=\mathcal{O}(n^{\alpha-1}) $, thus we can choose some $ \varepsilon\in (0,1) $ such that 
	\begin{align}\label{error-max-max}
		\begin{split}
			&\frac{1}{4}\|\widehat{e}_k\|^2+\Delta t\sum_{n=1}^{k}\|\widehat{e}_n\|_a^2+\Delta t^{1+\alpha}\sum_{n=1}^{k}\sum_{p=0}^{n}\mathcal{X}_p^{(\alpha,\lambda)}\mathscr{B}(\widehat{e}_{n-p},\widehat{e}_n)\\
			&\qquad\le CT\max_{2\le n\le N}\max_{1\le p\le n}\|\widehat{e}_{n-p}-\widetilde{e}_{n-1,n-p}\|_a^2+\frac{1}{2}\Delta t\sum_{n=1}^{k}\|\widehat{e}_n\|_a^2.
		\end{split}
	\end{align}
	Finally, we apply \Cref{maxmax,kernel-positive}  to \eqref{error-max-max} to  arrive  at
	\begin{align*}
		\|\widehat{e}_N\|\le C\sqrt{T\sigma(A)}(T_{sv}+1)\mathtt{tol}.
	\end{align*}
\end{proof}
Combining \Cref{singular-lemme-error-fem}  with \Cref{singular-isvd-error}, we use triangle inequality to finish  proof of \Cref{main_result 2}.
\section{Numerical experiments}\label{Numerical_experiments}
This section entails numerous numerical experiments conducted to demonstrate the efficiency and accuracy of our scheme, denoted as \eqref{ISVD_eq1}. In each experiment, we set $ \Omega=(0,1)\times (0,1) $ and $ T=1 $. We employ linear finite elements for spatial discretization and utilize second-order numerical schemes such as the Crank-Nicolson scheme and BDF2 scheme for time discretization. Our numerical experiments cover two aspects, including the convergence rate for both non-singular and singular kernels.

\begin{example}\label{example1}
	{In this particular example, we consider a scenario where the kernel is non-singular, and the time domain is uniformly partitioned. Here, $ h $ denotes the mesh size, while $ \Delta t $ represents the time step. For this case, we opt for $ h=\Delta t $, and the equation is specified as follows:}
	\begin{align*}
		u_t-\Delta u-\int_0^t K(t-s)\Delta u(s)ds=f(x,t)
	\end{align*}
	with 
	\begin{align*}
		u(x,y,t)=x(1-x)y(1-y)t,\ K(t)=\ln{(1+t)}.
	\end{align*}
	\begin{table}[h]
		\centering
		\begin{tabular}{c|c|c|c|c|c}
			\Xhline{1pt}
			$\frac{h}{\sqrt{2}}$  & $ \|u_h^N-u(T)\| $ & rate & $ \|\widehat u_h^N-u(T)\| $ & rate & $ \|u_h^N-\widehat{u}_h^N\| $\\ \Xhline{1pt}
			$ 1/2^3 $  & 1.38E-03      & -   & 1.38E-03     & -   & 3.18E-17       \\ \hline
			$ 1/2^4 $  & 3.50E-04     & 1.98    & 3.50E-04      &1.96   & 2.47E-15      \\ \hline
			$ 1/2^5 $ & 8.79E-05       & 2.00   & 8.79E-05      & 1.99   & 3.20E-14      \\ \hline
			$ 1/2^6 $  & 2.20E-05     & 2.00    & 2.20E-05      & 2.00   & 3.58E-14      \\ \hline
			$ 1/2^7 $  & 5.50E-06      & 2.00    & 5.50E-06      & 2.00    & 6.87E-14       \\ \hline
			$ 1/2^8 $  & 1.37E-06       & 2.00    & 1.37E-06    & 2.00    & 2.04E-14      \\ \hline
			$ 1/2^9 $  & 3.44E-07       & 2.00    & 3.44E-07      & 2.00   & 3.73E-15     \\ \hline
			$ 1/2^{10} $ & 8.59E-08       & 2.00    & 8.59E-08      & 2.00   & 9.79E-16       \\ \Xhline{1pt}
		\end{tabular}
		\caption{The convergence rates of $\| u_h^N-u(T)\|$, $\|\widehat u_h^N-u(T) \|$  and $ \|u_h^N-\widehat{u}_h^N\| $ for $\Delta t = h$  with $ K(t)=\ln{(1+t)} $.}
		\label{new-table1}
	\end{table}
	
	In this particular example, we test our approach using $ \mathtt{tol}=10^{-12} $ as selected in our new method. We record the error in $L^2$ norm between the finite element solution, the exact solution, and the solution obtained via our new scheme and compare the convergence order of the two algorithms  in \cref{new-table1}. Furthermore, we 
	compare the wall time and memory  cost for both the finite element method and our approach  and plot four figures for intuition in \cref{figure--1}.
\end{example}
\begin{example}\label{example--2}
	
	In this example, we evaluate the performance of our new method concerning the singular kernel given by
	\begin{align*}
		K(t)=e^{-\lambda t}\frac{1}{\Gamma(\alpha)}t^{\alpha-1}.
	\end{align*}
	We utilize linear finite elements and the BDF2 scheme for spatial and time discretization, respectively. The equation under consideration is expressed as follows:
	\begin{align*}
		u_t-\Delta u-\int_0^t K(t-s)\Delta u(s)\ {\rm d}s=f,
	\end{align*}
	where
	\begin{align*}
		u=-\frac{t^{2+\alpha}}{\Gamma(3+\alpha)}e^{-\lambda t}\sin(\pi x)\sin(\pi y),
	\end{align*}
	and the source term $f$ is defined as:
	\begin{align*}
		f=\left(\frac{\lambda t^{2+\alpha}}{\Gamma(3+\alpha)}-\frac{t^{1+\alpha}}{\Gamma(2+\alpha)}-\frac{2\pi^2t^{2\alpha+2}}{\Gamma(3+\alpha)}-\frac{2\pi^2t^{\alpha+2}}{\Gamma(3+\alpha)}\right)e^{-\lambda t}\sin(\pi x)\sin(\pi y),
	\end{align*}
	with parameters
	\begin{align*}
		\alpha=0.8, \  \lambda=0.2.
	\end{align*}
	\begin{table}[H]\label{new-table 3}
		\centering
		\begin{tabular}{c|c|c|c|c|c}
			\Xhline{1pt}
			$\frac{h}{\sqrt{2}}$  & $ \|u_h^N-u(T)\| $ & rate & $ \|\widehat u_h^N-u(T)\| $ & rate & $ \|u_h^N-\widehat{u}_h^N\| $\\ \Xhline{1pt}
			$ 1/2^3 $  & 4.00E-03     & -   & 4.00E-03      & -   & 9.26E-15       \\ \hline
			$ 1/2^4 $  & 1.00E-03    & 2.00    & 1.00E-03        &2.00   & 8.55E-15     \\ \hline
			$ 1/2^5 $ & 2.61E-04      & 1.94    & 2.61E-04       & 1.94   & 1.24E-14     \\ \hline
			$ 1/2^6 $  &6.55E-05    & 2.00   & 6.55E-05      & 2.00  & 4.73E-15    \\ \hline
			$ 1/2^7 $  & 1.64E-05    & 2.00   &1.64E-05      & 2.00    & 2.67E-15      \\ \hline
			$ 1/2^8 $  & 4.10E-06       & 2.00   & 4.10E-06      &2.00 & 3.98E-16     \\ \hline
			$ 1/2^9 $  & 1.03E-06     & 2.00    & 1.03E-06      & 2.00  & 9.18E-16      \\ \hline
			$ 1/2^{10} $ & 2.57E-07      & 2.00    & 2.57E-07       & 2.00   & 7.77E-16       \\ \Xhline{1pt}
		\end{tabular}
		\caption{The convergence rates of $\| u_h^N-u(T)\|$, $\|\widehat u_h^N-u(T) \|$  and $ \|u_h^N-\widehat{u}_h^N\| $ for $\Delta t = h$  with parameters  $ \alpha=0.8,\lambda=0.2 $.}
	\end{table}
	
	Referencing \cite{QiuWenlin2023AFEG}, it is established that this particular example attains a second-order convergence rate in both space and time. Additionally, we conduct tests measuring the $ L^2 $ error between the exact solution and the solution acquired through our new approach, as well as the $ L^2 $ error between the solutions obtained via the two methods. The numerical experiments demonstrate that our new method maintains an error bound nearly equivalent to that of the standard finite element approach.
	
	Additionally, we also compare the wall time and memory  for both the finite element method and our approach for the case that kernel is singular ,  and plot two figures for intuition in \cref{figure--2}. It is evident that our new approach is more efficient, especially when the time step and mesh size are small. This suggests that our scheme performs well for large-scale problems.
	
\end{example}

\begin{figure}[h]
	\centering
	\begin{minipage}{0.49\linewidth}
		\centering
		\includegraphics[width=0.9\linewidth]{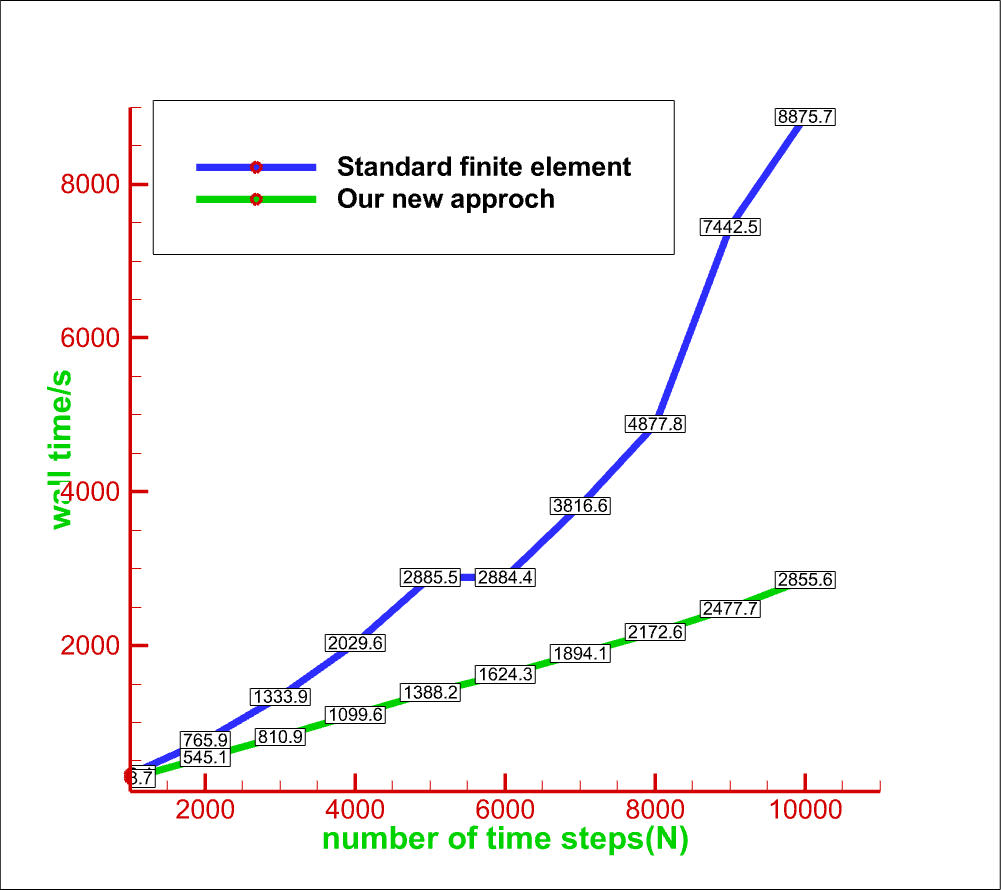}
	\end{minipage}
	%\qquad
	\begin{minipage}{0.49\linewidth}
		\centering
		\includegraphics[width=0.9\linewidth]{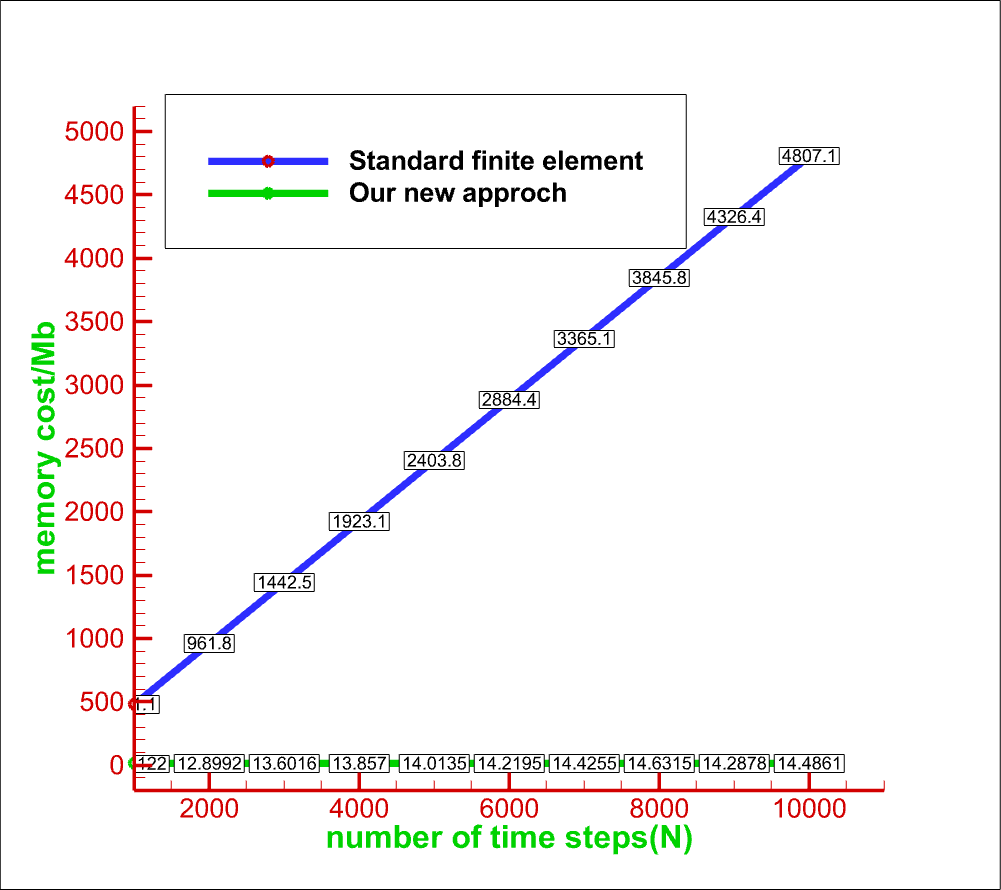}
	\end{minipage}
	\caption{A comparison of wall time and memory costs between the two algorithms is conducted for various time steps and $ K(t)=\log {(1+t)} $, specifically when $h=1/250$ and $\Delta t=10^{-4}$.}
	\label{figure--1}
\end{figure}
\begin{figure}[h]
	\centering
	\begin{minipage}{0.49\linewidth}
		\centering
		\includegraphics[width=0.9\linewidth]{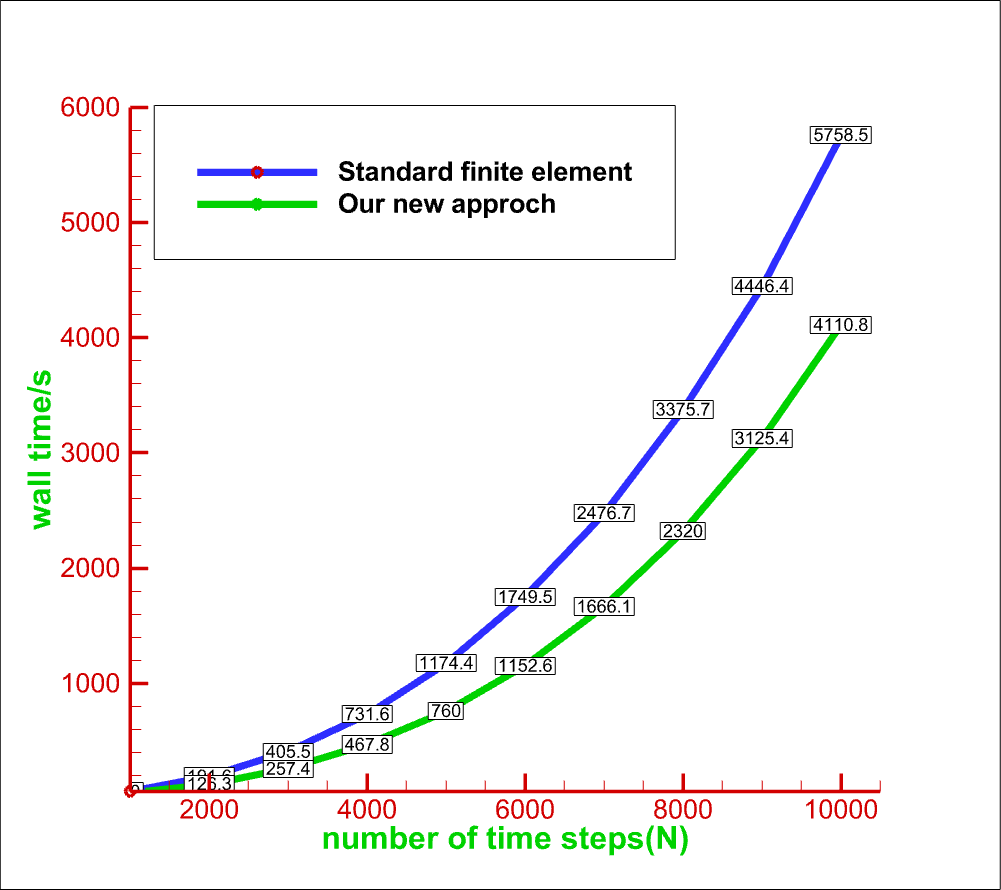}
	\end{minipage}
	%\qquad
	\begin{minipage}{0.49\linewidth}
		\centering
		\includegraphics[width=0.9\linewidth]{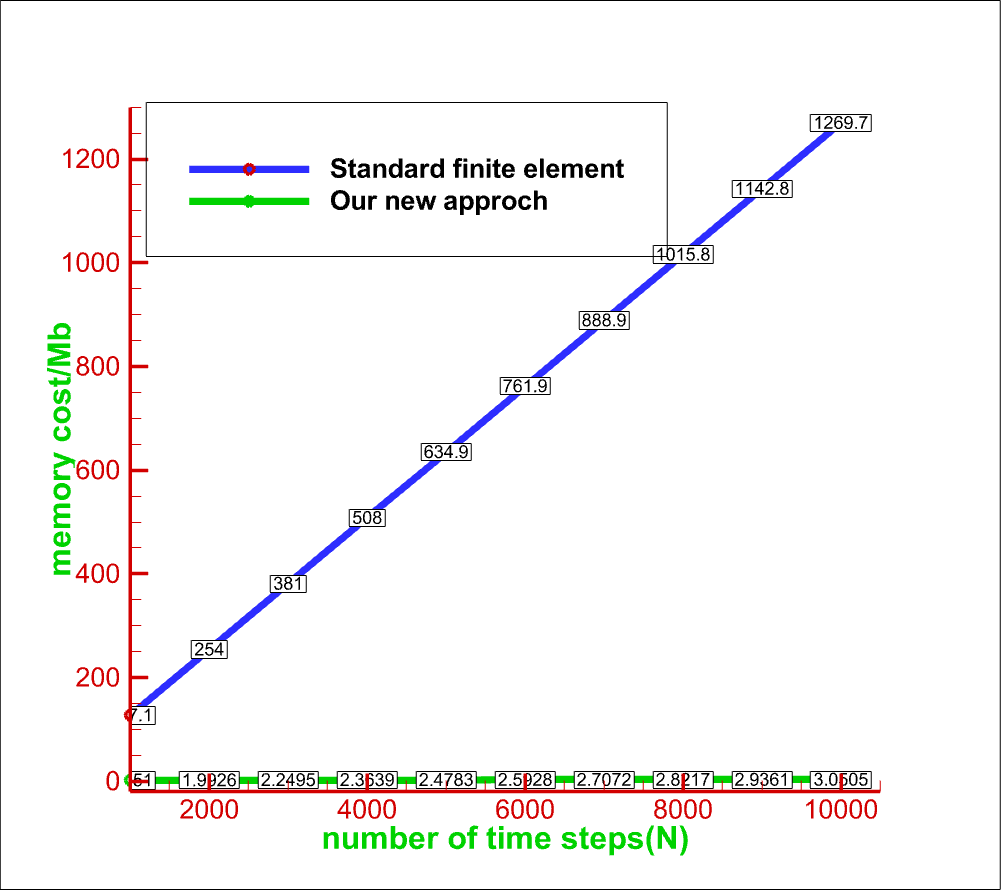}
	\end{minipage}
	\caption{A comparison of wall time and memory costs between the two algorithms is conducted for various time steps when $ \displaystyle K(t)=e^{-\lambda t}\frac{1}{\Gamma(\alpha)}t^{\alpha-1} $ with $ \alpha=0.8,\lambda=0.2 $, specifically when $h=1/128$ and $\Delta t=10^{-4}$.}
	\label{figure--2}
\end{figure}
\begin{example}
	
	In this example, we consider a more general case, which is an extension of \Cref{example--2}. We consider the following non-fickian model with variable-order
	Caputo fractional derivative:
	\begin{align*}
		u_t+\mathcal{A} u+\frac{1}{\Gamma(1-\alpha(t))}\int_0^t\frac{1}{(t-s)^{\alpha(t)}}\mathcal{B}u(s)ds=f(x,t).
	\end{align*}
	where $ \mathcal{A}=\mathcal{B}=-\Delta $ and  exact solution $ u $ is: 
	\begin{align*}
		u&=xy(1-x)(1-y)t, 
	\end{align*}
	the  sourse term $ f $  is defined as follows :
	\begin{align*}
		f(x,y,t)&=xy(1-x)(1-y)+2t(x+y-x^2-y^2)+\frac{2t^{2-\alpha(t)}}{\Gamma(3-\alpha(t))}(x+y-x^2-y^2),\\
		\alpha(t)&=\frac{1}{2}+\frac{1}{4}\sin{(5t)}.
	\end{align*}
	We use linear finite elements for spatial discretization and Back Euler method for temporal discretization. For approximation of Caputo fractional derivative , we use $ L^1 $ convolution quadrature  method from \cite{ShenS.2012Ntft}:
	\begin{align*}
		&\frac{1}{\Gamma(1-\alpha(t_n))}\int_{0}^{t_n}\frac{1}{(t_n-s)^{\alpha(t_n)}}(\nabla u(s),\nabla v_h)ds\\
		&\approx \frac{1}{\Gamma(1-\alpha_n)}\sum_{j=1}^{n}\int_{t_{j-1}}^{t_j}\frac{1}{(t_n-s)^{\alpha_n}}ds(\nabla u_h^j,\nabla v_h)\\
		&=\sum_{j=1}^{n}\frac{1}{\Gamma(2-\alpha_n)}[(t_n-t_{j-1})^{1-\alpha_n}-(t_n-t_j)^{1-\alpha_n}](\nabla u_h^j,\nabla v_h)\\
		&=\sum_{j=1}^{n}\beta_{n,j}(\nabla u_h^j,\nabla v_h),
	\end{align*}
	where
	\begin{align*}
		\beta_{n,j}=\frac{1}{\Gamma(2-\alpha_n)}[(t_n-t_{j-1})^{1-\alpha_n}-(t_n-t_j)^{1-\alpha_n}],\ \alpha_n=\alpha(t_n).
	\end{align*}
	Different from other fast algorithms in many literatures  , our innovative method stems from data science perspective and attains uniformity in spite of the form of kernel $ K(t) $. In this example, we compare the $ L^2 $  error norm between the finite element solution, the exact solution , and the solution obtained via our novel method and record wall time comsumed by two methods in \cref{new-table 4}. It is well known that the storage cost   and computational cost of standard $ L^1 $ method are $ \mathcal{O}(mn) $ and $ \mathcal{O}(mn^2) $, while the storage and computational cost of our new method are $ \mathcal{O}((m+n)r) $ and $ \mathcal{O}(mnr+rn^2) $ . During this example, we use the  size of SVD matrix $ \mathtt{size(\Sigma,1)}$ to    evaluate $ r $ approximately and record  change of $ r $ as time step $ N $ increases in \cref{Rank-Compare}, which suggests that  rank of data is more less than time step $N  $ . Numerical experiment demonstrates our innovative method keeps an almostly same error as standard $ L^1 $ method. Rigorous theoretical analysis for variable-order Caputo fractional derivative will be carried out in future investigation. 
	\begin{table}[H]
		\centering
		\begin{tabular}{c|c|c|c|c|c}
			\Xhline{1pt}
			$\frac{h}{\sqrt{2}}$  & $ \|u_h^N-u(T)\| $  & $ \|\widehat u_h^N-u(T)\| $  & $ \|u_h^N-\widehat{u}_h^N\| $ &time 1/s &time 2/s \\ \Xhline{1pt}
			$ 1/2^2 $  & 5.37E-03     &5.37E-03   & 9.65E-14     & 1.19   & 1.44       \\ \hline
			$ 1/2^3 $  & 1.42E-03     & 1.42E-03   & 2.33E-14      & 2.68   & 2.43       \\ \hline
			$ 1/2^4 $  & 3.61E-04    & 3.61E-04    & 3.73E-14        &7.31  & 2.97     \\ \hline
			$ 1/2^5 $ & 9.22E-05      & 9.22E-05   & 2.07E-14       & 25.59   & 7.01     \\ \hline
			$ 1/2^6 $  &2.47E-05    & 2.47E-05   & 1.13E-14      & 97.01  & 25.21    \\ \hline
			$ 1/2^7 $  & 7.82E-06    & 7.82E-06   &6.09E-15     & 395.70    & 139.61      \\  \hline
			$ 1/2^8 $  &3.69E-06    & 3.69E-06   & 3.16E-15     & 1572.78  & 532.11    \\ 
			\Xhline{1pt}
		\end{tabular}
		\caption{The $ L^2 $ error norm of $\| u_h^N-u(T)\|$, $\|\widehat u_h^N-u(T) \|$  ,$ \|u_h^N-\widehat{u}_h^N\| $   and time comparsion for $\Delta t = \frac{1}{4000}$ : time 1  and time 2 are wal time consumed by traditional method and our innovative method. }
		\label{new-table 4}
	\end{table}
	\begin{figure}
		\centering %表示居中
		\includegraphics[height=4.5cm,width=7.5cm]{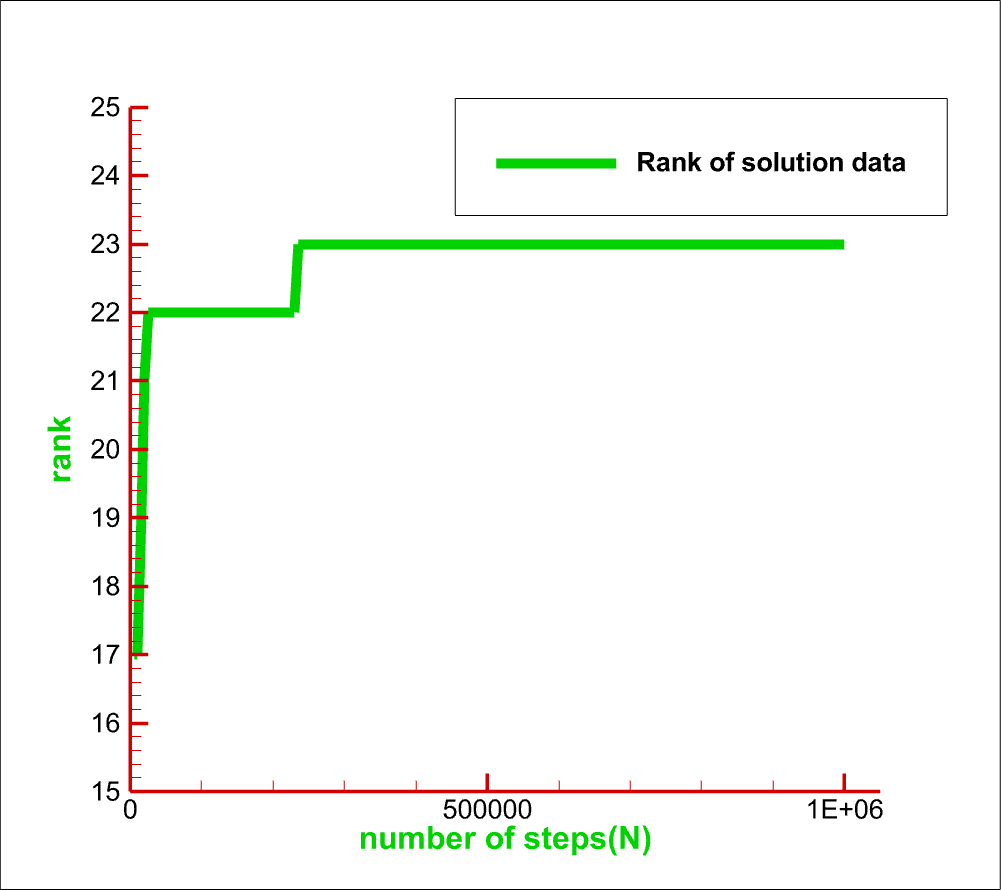}
		\caption{rank changes with respect to  number of time steps $ N $}
		\label{Rank-Compare}
	\end{figure}

\end{example}

\section{Conclusion}
In this paper, we present a novel and efficient algorithm for resolving Non-Fickian flows. Our method draws inspiration from the incremental Singular Value Decomposition (SVD) technique commonly employed in data science. Notably, this approach exclusively operates on data intrinsic to the problem at hand. This adaptability leads us to posit its suitability for a range of significant models that encompass a memory component. A couple of such models include the Debye memory-enhanced Maxwell's equations \cite{MR2747091}. These promising avenues form the focus of our upcoming research endeavors.

\bibliographystyle{siam}
\bibliography{mybib}

\begin{thebibliography}{10}

\bibitem{MR1686149}
{\sc W.~Allegretto, Y.~Lin, and A.~Zhou}, {\em Long-time stability of finite
  element approximations for parabolic equations with memory}, Numer. Methods
  Partial Differential Equations, 15 (1999), pp.~333--354.

\bibitem{MR3778339}
{\sc S.~Barbeiro, S.~G. Bardeji, J.~A. Ferreira, and L.~Pinto}, {\em
  Non-{F}ickian convection-diffusion models in porous media}, Numer. Math., 138
  (2018), pp.~869--904.

\bibitem{MR2670111}
{\sc N.~Bauermeister and S.~Shaw}, {\em Finite-element approximation of
  non-{F}ickian polymer diffusion}, IMA J. Numer. Anal., 30 (2010),
  pp.~702--730.

\bibitem{brand2002incremental}
{\sc M.~Brand}, {\em Incremental singular value decomposition of uncertain data
  with missing values}, in Computer Vision—ECCV 2002: 7th European Conference
  on Computer Vision Copenhagen, Denmark, May 28--31, 2002 Proceedings, Part I
  7, Springer, 2002, pp.~707--720.

\bibitem{MR2214744}
\leavevmode\vrule height 2pt depth -1.6pt width 23pt, {\em Fast low-rank
  modifications of the thin singular value decomposition}, Linear Algebra
  Appl., 415 (2006), pp.~20--30.

\bibitem{ChenC.1992FEAo}
{\sc C.~Chen, V.~Thomée, and L.~B. Wahlbin}, {\em Finite element approximation
  of a parabolic integro-differential equation with a weakly singular kernel},
  Mathematics of computation, 58 (1992), pp.~587--602.

\bibitem{MR3614283}
{\sc M.~Chen and W.~Deng}, {\em A second-order accurate numerical method for
  the space-time tempered fractional diffusion-wave equation}, Appl. Math.
  Lett., 68 (2017), pp.~87--93.

\bibitem{ChenYanping2012Read}
{\sc Y.~Chen, Y.~Huang, and T.~Hou}, {\em Richardson extrapolation and defect
  correction of mixed finite element methods for elliptic optimal control
  problems}, Journal of the Korean Mathematical Society, 49 (2012),
  pp.~549--569.

\bibitem{EduardoCuesta2006Cqtd}
{\sc E.~Cuesta, C.~Lubich, and C.~Palencia}, {\em Convolution quadrature time
  discretization of fractional diffusion-wave equations}, Mathematics of
  computation, 75 (2006), pp.~673--696.

\bibitem{MR1951906}
{\sc R.~E. Ewing, Y.~Lin, T.~Sun, J.~Wang, and S.~Zhang}, {\em Sharp
  {$L^2$}-error estimates and superconvergence of mixed finite element methods
  for non-{F}ickian flows in porous media}, SIAM J. Numer. Anal., 40 (2002),
  pp.~1538--1560.

\bibitem{MR4017489}
{\sc H.~Fareed and J.~R. Singler}, {\em Error analysis of an incremental proper
  orthogonal decomposition algorithm for {PDE} simulation data}, J. Comput.
  Appl. Math., 368 (2020), pp.~112525, 14.

\bibitem{MR3775096}
{\sc H.~Fareed, J.~R. Singler, Y.~Zhang, and J.~Shen}, {\em Incremental proper
  orthogonal decomposition for {PDE} simulation data}, Comput. Math. Appl., 75
  (2018), pp.~1942--1960.

\bibitem{MR3557149}
{\sc J.~A. Ferreira and L.~Pinto}, {\em An integro-differential model for
  non-{F}ickian tracer transport in porous media: validation and numerical
  simulation}, Math. Methods Appl. Sci., 39 (2016), pp.~4736--4749.

\bibitem{frolov2017tensor}
{\sc E.~Frolov and I.~Oseledets}, {\em Tensor methods and recommender systems},
  Wiley Interdisciplinary Reviews: Data Mining and Knowledge Discovery, 7
  (2017), p.~e1201.

\bibitem{MR2167744}
{\sc L.~Giraud, J.~Langou, and M.~Rozloznik}, {\em The loss of orthogonality in
  the {G}ram-{S}chmidt orthogonalization process}, Comput. Math. Appl., 50
  (2005), pp.~1069--1075.

\bibitem{GuoLing2019Emmf}
{\sc L.~Guo, F.~Zeng, I.~Turner, K.~Burrage, and G.~E. Karniadakis}, {\em
  Efficient multistep methods for tempered fractional calculus: Algorithms and
  simulations}, SIAM journal on scientific computing, 41 (2019),
  pp.~A2510--A2535.

\bibitem{GuoYingwen2022Ceaf}
{\sc Y.~Guo and Y.~He}, {\em Crank–nicolson extrapolation and finite element
  method for the oldroyd fluid with the midpoint rule}, Journal of
  computational and applied mathematics, 415 (2022), p.~114453.

\bibitem{hemati2014dynamic}
{\sc M.~S. Hemati, M.~O. Williams, and C.~W. Rowley}, {\em Dynamic mode
  decomposition for large and streaming datasets}, Physics of Fluids, 26
  (2014).

\bibitem{HuangYuxiang2022Eeot}
{\sc Y.~Huang, F.~Zeng, and L.~Guo}, {\em Error estimate of the fast l1 method
  for time-fractional subdiffusion equations}, Applied mathematics letters, 133
  (2022), p.~108288.

\bibitem{MR2447252}
{\sc S.~Jia, D.~Li, and S.~Zhang}, {\em Asymptotic expansions and {R}ichardson
  extrapolation of approximate solutions for integro-differential equations by
  mixed finite element methods}, Adv. Comput. Math., 29 (2008), pp.~337--356.

\bibitem{MR3639246}
{\sc S.~Jiang, J.~Zhang, Q.~Zhang, and Z.~Zhang}, {\em Fast evaluation of the
  {C}aputo fractional derivative and its applications to fractional diffusion
  equations}, Commun. Comput. Phys., 21 (2017), pp.~650--678.

\bibitem{MR1713237}
{\sc Z.~Jiang}, {\em {$L^\infty(L^2)$} and {$L^\infty(L^\infty)$} error
  estimates for mixed methods for integro-differential equations of parabolic
  type}, M2AN Math. Model. Numer. Anal., 33 (1999), pp.~531--546.

\bibitem{J.-P.Kauthen1997Ccat}
{\sc J.-P. Kauthen and H.~Brunner}, {\em Continuous collocation approximations
  to solutions of first kind volterra equations}, Mathematics of computation,
  66 (1997), pp.~1441--1459.

\bibitem{KwonKiwoon2003Apmf}
{\sc K.~Kwon and D.~Sheen}, {\em A parallel method for the numerical solution
  of integro-differential equation with positive memory}, Computer methods in
  applied mechanics and engineering, 192 (2003), pp.~4641--4658.

\bibitem{LiXiaoli2016Atbf}
{\sc X.~Li and H.~Rui}, {\em A two-grid block-centered finite difference method
  for nonlinear non-fickian flow model}, Applied mathematics and computation,
  281 (2016), pp.~300--313.

\bibitem{MR3084175}
{\sc H.~Liang and H.~Brunner}, {\em Integral-algebraic equations: theory of
  collocation methods {I}}, SIAM J. Numer. Anal., 51 (2013), pp.~2238--2259.

\bibitem{LIANGHUI2016IETO}
{\sc H.~LIANG and H.~BRUNNER}, {\em Integral-algebraic equations: Theory of
  collocation methods ii}, SIAM journal on numerical analysis, 54 (2016),
  pp.~2640--2663.

\bibitem{LiangHui2019Tcoc}
{\sc H.~Liang and H.~Brunner}, {\em The convergence of collocation solutions in
  continuous piecewise polynomial spaces for weakly singular volterra integral
  equations}, SIAM journal on numerical analysis, 57 (2019), pp.~1875--1896.

\bibitem{LiangHui2020Cmfi}
\leavevmode\vrule height 2pt depth -1.6pt width 23pt, {\em Collocation methods
  for integro-differential algebraic equations with index 1}, IMA journal of
  numerical analysis, 40 (2020), pp.~850--885.

\bibitem{MR1034918}
{\sc J.~C. L\'{o}pez~Marcos}, {\em A difference scheme for a nonlinear partial
  integrodifferential equation}, SIAM J. Numer. Anal., 27 (1990), pp.~20--31.

\bibitem{LubichC.1988Cqad}
{\sc C.~Lubich}, {\em Convolution quadrature and discretized operational
  calculus. i}, Numerische Mathematik, 52 (1988), pp.~129--145.

\bibitem{MR0923707}
\leavevmode\vrule height 2pt depth -1.6pt width 23pt, {\em Convolution
  quadrature and discretized operational calculus. {I}}, Numer. Math., 52
  (1988), pp.~129--145.

\bibitem{LuoMan2013CQBN}
{\sc M.~Luo, D.~Xu, and L.~Li}, {\em Crank-nicolson quasi-wavelet based
  numerical method for volterra integro-differential equations on unbounded
  spatial domains}, East Asian journal on applied mathematics, 3 (2013),
  pp.~283--292.

\bibitem{MR3335208}
{\sc G.~Murali Mohan~Reddy and R.~K. Sinha}, {\em Ritz-{V}olterra
  reconstructions and {\it a posteriori} error analysis of finite element
  method for parabolic integro-differential equations}, IMA J. Numer. Anal., 35
  (2015), pp.~341--371.

\bibitem{MR2608473}
{\sc K.~Mustapha and H.~Mustapha}, {\em A second-order accurate numerical
  method for a semilinear integro-differential equation with a weakly singular
  kernel}, IMA J. Numer. Anal., 30 (2010), pp.~555--578.

\bibitem{OstermannAlexander2023EERM}
{\sc A.~Ostermann, F.~Saedpanah, and N.~Vaisi}, {\em Explicit exponential
  runge-kutta methods for semilinear integro-differential equations}, SIAM
  journal on numerical analysis, 61 (2023), pp.~1405--1425.

\bibitem{MR3594691}
{\sc G.~M. Oxberry, T.~Kostova-Vassilevska, W.~Arrighi, and K.~Chand}, {\em
  Limited-memory adaptive snapshot selection for proper orthogonal
  decomposition}, Internat. J. Numer. Methods Engrg., 109 (2017), pp.~198--217.

\bibitem{MR1897408}
{\sc A.~K. Pani and G.~Fairweather}, {\em {$H^1$}-{G}alerkin mixed finite
  element methods for parabolic partial integro-differential equations}, IMA J.
  Numer. Anal., 22 (2002), pp.~231--252.

\bibitem{MR2580558}
{\sc A.~K. Pani, G.~Fairweather, and R.~I. Fernandes}, {\em A{DI} orthogonal
  spline collocation methods for parabolic partial integro-differential
  equations}, IMA J. Numer. Anal., 30 (2010), pp.~248--276.

\bibitem{MR1393903}
{\sc A.~K. Pani and T.~E. Peterson}, {\em Finite element methods with numerical
  quadrature for parabolic integrodifferential equations}, SIAM J. Numer.
  Anal., 33 (1996), pp.~1084--1105.

\bibitem{QiuWenlin2023AFEG}
{\sc W.~Qiu, G.~Fairweather, X.~Yang, and H.~Zhang}, {\em Adi finite element
  galerkin methods for two-dimensional tempered fractional integro-differential
  equations}, Calcolo, 60 (2023).

\bibitem{QiuWenlin2020Atta}
{\sc W.~Qiu, D.~Xu, J.~Guo, and J.~Zhou}, {\em A time two-grid algorithm based
  on finite difference method for the two-dimensional nonlinear time-fractional
  mobile/immobile transport model}, Numerical algorithms, 85 (2020),
  pp.~39--58.

\bibitem{MR2272610}
{\sc B.~Rivi\`ere and S.~Shaw}, {\em Discontinuous {G}alerkin finite element
  approximation of nonlinear non-{F}ickian diffusion in viscoelastic polymers},
  SIAM J. Numer. Anal., 44 (2006), pp.~2650--2670.

\bibitem{ross2008incremental}
{\sc D.~A. Ross, J.~Lim, R.-S. Lin, and M.-H. Yang}, {\em Incremental learning
  for robust visual tracking}, International journal of computer vision, 77
  (2008), pp.~125--141.

\bibitem{MR2231714}
{\sc A.~Sch\"{a}dle, M.~L\'{o}pez-Fern\'{a}ndez, and C.~Lubich}, {\em Fast and
  oblivious convolution quadrature}, SIAM J. Sci. Comput., 28 (2006),
  pp.~421--438.

\bibitem{MR2747091}
{\sc S.~Shaw}, {\em Finite element approximation of {M}axwell's equations with
  {D}ebye memory}, Adv. Numer. Anal.,  (2010), pp.~Art. ID 923832, 28.

\bibitem{ShenS.2012Ntft}
{\sc S.~Shen, F.~Liu, J.~Chen, I.~Turner, and V.~Anh}, {\em Numerical
  techniques for the variable order time fractional diffusion equation},
  Applied mathematics and computation, 218 (2012), pp.~10861--10870.

\bibitem{MR2206438}
{\sc R.~K. Sinha, R.~E. Ewing, and R.~D. Lazarov}, {\em Some new error
  estimates of a semidiscrete finite volume element method for a parabolic
  integro-differential equation with nonsmooth initial data}, SIAM J. Numer.
  Anal., 43 (2006), pp.~2320--2343.

\bibitem{MR2551194}
\leavevmode\vrule height 2pt depth -1.6pt width 23pt, {\em Mixed finite element
  approximations of parabolic integro-differential equations with nonsmooth
  initial data}, SIAM J. Numer. Anal., 47 (2009), pp.~3269--3292.

\bibitem{MR859017}
{\sc I.~H. Sloan and V.~Thom\'{e}e}, {\em Time discretization of an
  integro-differential equation of parabolic type}, SIAM J. Numer. Anal., 23
  (1986), pp.~1052--1061.

\bibitem{MR1220827}
{\sc V.~Thom\'{e}e and L.~B. Wahlbin}, {\em Long-time numerical solution of a
  parabolic equation with memory}, Math. Comp., 62 (1994), pp.~477--496.

\bibitem{WangWansheng2018Teaf}
{\sc W.~Wang and Q.~Hong}, {\em Two-grid economical algorithms for parabolic
  integro-differential equations with nonlinear memory}, 2018.

\bibitem{MR3958333}
\leavevmode\vrule height 2pt depth -1.6pt width 23pt, {\em Two-grid economical
  algorithms for parabolic integro-differential equations with nonlinear
  memory}, Appl. Numer. Math., 142 (2019), pp.~28--46.

\bibitem{MR1163355}
{\sc Y.~Yan and G.~Fairweather}, {\em Orthogonal spline collocation methods for
  some partial integrodifferential equations}, SIAM J. Numer. Anal., 29 (1992),
  pp.~755--768.

\bibitem{YangXuehua2013Cmfs}
{\sc X.~Yang, D.~Xu, and H.~Zhang}, {\em Crank–nicolson/quasi-wavelets method
  for solving fourth order partial integro-differential equation with a weakly
  singular kernel}, Journal of computational physics, 234 (2013), pp.~317--329.

\bibitem{zhang2022answer}
{\sc Y.~Zhang}, {\em An answer to an open question in the incremental {SVD}},
  arXiv:2204.05398,  (2022).

\end{thebibliography}

\end{document}